\documentclass{amsart}

\usepackage{macros}
\standardsettings

\usepackage{wrapfig}
\usepackage{lscape}
\usepackage{rotating}
\usepackage{epstopdf}

\newcommand{\mult}{\times}

\newcommand{\Hyp}{\scrH}
\newcommand{\Sph}{\scrS}

\newcommand{\alphabet}{A}
\newcommand{\borel}{E}

\newcommand{\Index}{n}

\usepackage{stackrel}

\draftfalse

\theoremstyle{theorem}
\maketheorem{reduction}{Reduction}

\begin{document}
\title[Extremality and dynamically defined measures, II]{Extremality and dynamically defined measures, part II: Measures from conformal dynamical systems}

\authortushar\authorlior\authordavid\authormariusz


\begin{Abstract}
We present a new method of proving the Diophantine extremality of various dynamically defined measures, vastly expanding the class of measures known to be extremal. This generalizes and improves the celebrated theorem of Kleinbock and Margulis [Logarithm laws for flows on homogeneous spaces. {\it Invent. Math.} {\bf 138}(3) (1999), 451--494] resolving Sprind\v zuk's conjecture, as well as its extension by Kleinbock, Lindenstrauss, and Weiss [On fractal measures and Diophantine approximation. {\it Selecta Math.} {\bf 10} (2004), 479--523], hereafter abbreviated KLW. As applications we prove the extremality of all hyperbolic measures of smooth dynamical systems with sufficiently large Hausdorff dimension, and of the Patterson--Sullivan measures of all nonplanar geometrically finite groups. The key technical idea, which has led to a plethora of new applications, is a significant weakening of KLW's sufficient conditions for extremality.
In the first of this series of papers [Extremality and dynamically defined measures, part I: Diophantine properties of quasi-decaying measures. {\it Selecta Math.} {\bf 24}(3) (2018), 2165--2206], we introduce and develop a systematic account of two classes of measures, which we call {\it quasi-decaying} and {\it weakly quasi-decaying}. We prove that weak quasi-decay implies strong extremality in the matrix approximation framework, as well as proving the ``inherited exponent of irrationality'' version of this theorem.
In this paper, the second of the series, we establish sufficient conditions on various classes of conformal dynamical systems for their measures to be quasi-decaying. In particular, we prove the above-mentioned result about Patterson--Sullivan measures, and we show that equilibrium states (including conformal measures) of nonplanar infinite iterated function systems (including those which do not satisfy the open set condition) and rational functions are quasi-decaying. 
\end{Abstract}

\maketitle
\tableofcontents

\section{Introduction}

In this series of papers we address a central problem in the flourishing area of metric Diophantine approximation on manifolds and measures: an attempt to exhibit a possibly widest natural class of sets and measures for which most points are not very well approximable by ones with rational coordinates.

Fix $d\in\N$. The quality of rational approximations to a vector $\xx\in\R^d$ can be measured by its \emph{exponent of irrationality}, which is defined by the formula
\[
\omega(\xx) = \limsup_{\pp/q\in\Q^d} \frac{-\log\|\xx - \pp/q\|}{\log(q)},
\]
where the limsup is taken over any enumeration of $\Q^d$, and $\|\cdot\|$ is any norm on $\R^d$. Another interesting quantity is the \emph{exponent of multiplicative irrationality}, which is the number
\[
\omega_\mult(\xx) = \limsup_{\pp/q\in\Q^d} \frac{-\log\prod_{i = 1}^d |x_i - p_i/q|}{\log(q)}\cdot
\]
It follows from a pigeonhole argument that $\omega(\xx) \geq 1 + 1/d$ and $\omega_\mult(\xx) \geq d + 1$. A vector $\xx$ is said to be \emph{very well approximable} if $\omega(\xx) > 1 + 1/d$, and \emph{very well multiplicatively approximable} if $\omega_\mult(\xx) > d + 1$. We will denote the set of very well (multiplicatively) approximable vectors by $\mathrm{VW(M)A}_d$. It is well-known that $\VWA_d$ and $\VWMA_d$ are both Lebesgue nullsets of full Hausdorff dimension, and that $\VWA_d \subset \VWMA_d$.

A measure $\mu$ on $\R^d$ is \emph{extremal} if $\mu(\VWA_d) = 0$, and \emph{strongly extremal} if $\mu(\VWMA_d) = 0$. Extremality was first defined by V.~G.~Sprind\v zuk, who conjectured that the Lebesgue measure of any nondegenerate manifold is extremal. This conjecture was proven by D.~Y.~Kleinbock and G.~A.~Margulis \cite{KleinbockMargulis}, and later strengthened by D.~Y.~Kleinbock, E.~Lindenstrauss, and B.~Weiss (hereafter abbreviated ``KLW'') in \cite{KLW}, who considered a class of measures which they called ``friendly'' and showed that these measures are strongly extremal. However, their definition is somewhat rigid and many interesting measures, in particular ones coming from dynamics, do not satisfy their condition. In this paper, we study a much larger class of measures, which we call \emph{weakly quasi-decaying}, such that every weakly quasi-decaying measure is strongly extremal \cite[Corollary 1.8]{DFSU_GE1}. This class includes a subclass of \emph{quasi-decaying} measures, which are the analogue of KLW's ``absolutely friendly'' measures. (The terminology ``absolutely friendly'' was not used by KLW and first appeared in \cite{PollingtonVelani}; however, several theorems about absolute friendliness had already appeared in \cite{KLW} without using the terminology.)

In the previous paper (Part I), which has now appeared as \cite{DFSU_GE1}, we analyzed the basic properties of the quasi-decay and weak quasi-decay conditions, as well as proving that every weakly quasi-decaying measure is strongly extremal. In fact, we proved a more refined version of that theorem which we will not state here. We proved that every exact dimensional measure on $\R^d$ whose dimension $\delta$ satisfies $\delta > d - 1$ is quasi-decaying. This elementary result already provides many dynamical examples of quasi-decaying measures; for example, a theorem of F.~Hofbauer \cite{Hofbauer} states that any measure invariant under a piecewise smooth endomorphism of $[0,1]$ which has positive entropy is exact dimensional of positive dimension; since $d = 1$ in this setup, this implies that such measures are quasi-decaying. (See \6\ref{subsectionexamples} for more details.)

In the current paper (Part II) we continue with this theme, examining several classes of conformal dynamical systems to see under what circumstances the inequality $\delta > d - 1$ can be relaxed to a ``nonplanarity'' assumption. The examples we will consider are intended to support the thesis that ``any sufficiently non-pathological measure coming from dynamics or fractal geometry is quasi-decaying''. As a contrast, we also provide examples of non-extremal measures coming from dynamics, to gain a better insight as to which measures should be considered ``pathological''.

In future work, we hope to elucidate the scope of our thesis that ``almost any measure from dynamics and/or fractal geometry is quasi-decaying'' by continuing to find examples of quasi-decaying measures, including random measures and measures invariant under diffeomorphisms.

There is a wide range of potential research directions that readers may follow up our paper with. For instance, a natural next step would be to move beyond the conformal realm and determine the extremality of self-affine fractal sets. One possibility would be to study sponges, e.g. as in \cite{DasSimmons1}.

\begin{convention}
The symbols $\lesssim$, $\gtrsim$, and $\asymp$ will denote coarse asymptotics; a subscript of $\plus$ indicates that the asymptotic is additive, and a subscript of $\times$ indicates that it is multiplicative. For example, $A\lesssim_{\times,K} B$ means that there exists a constant $C > 0$ (the \emph{implied constant}), depending only on $K$, such that $A\leq C B$. $A\lesssim_{\plus,\times}B$ means that there exist constants $C_1,C_2 > 0$ so that $A\leq C_1 B + C_2$. In general, dependence of the implied constant(s) on universal objects will be omitted from the notation.

If $\mu$ and $\nu$ are measures, then $\nu\lesssim_\times\mu$ means that there exists a constant $C > 0$ such that $\nu\leq C\mu$.
\end{convention}

\begin{convention}
In this paper, all measures and sets are assumed to be Borel, and measures are assumed to be locally finite.
\end{convention}

\begin{convention}
For $S\subset\R^d$ and $\rho \geq 0$, $\thickvar S\rho = \{x\in\R^d:\dist(x,S) \leq \rho\}$ is the closed $\rho$-thickening of $S$, and $\thickopenvar{S}{\rho} = \{x\in\R^d:\dist(x,S) < \rho\}$ is the open $\rho$-thickening of $S$.
\end{convention}

\begin{convention}
$A\wedge B$ and $A\vee B$ denote the minimum and maximum of $A$ and $B$, respectively.
\end{convention}

\begin{convention}
\label{conventioniverson}
We use the Iverson bracket notation $[\text{statement}] = \begin{cases} 1 & \text{statement true} \\ 0 & \text{statement false} \end{cases}$.
\end{convention}

\begin{convention}
The image of a measure $\mu$ under a map $f$ is denoted $f_*[\mu] := \mu\circ f^{-1}$.
\end{convention}

\begin{convention}
$\Hyp$ denotes the collection of affine hyperplanes in $\R^d$.
\end{convention}

\begin{convention}
The symbol $\triangleleft$ will be used to indicate the end of a nested proof.
\end{convention}

\begin{convention}
The Hausdorff dimension of a set $S$ is denoted by $\HD(S)$, and $\scrH^s(S)$ denotes the $s$-dimensional Hausdorff measure of $S$.
\end{convention}

{\bf Acknowledgements.} The first-named author was supported in part by a 2014-2015 Faculty Research Grant from the University of Wisconsin--La Crosse. The second-named author was supported in part by the Simons Foundation grant \#245708. The third-named author was supported in part by the EPSRC Programme Grant EP/J018260/1, and by a fellowship from the Royal Society. The fourth-named author was supported in part by the NSF grant DMS-1361677. We thank the referee for useful comments and catching a number of typos that helped improve the readability, and also for asking questions that led to the inclusion of Appendix \ref{appendix}.

\subsection{Four conditions which imply strong extremality}
In \cite[\61.1]{DFSU_GE1}, we introduced the notions of quasi-decaying measures and weakly quasi-decaying measures, and compared them with KLW's notions of friendly and absolutely friendly measures. While we recall the definitions of the four classes of measures here, we refer to Part I for a more detailed discussion. 

\begin{definition}
\label{definitionabsolutelyfriendly}
Let $\mu$ be a measure on an open set $U\subset\R^d$, and let $\Supp(\mu)$ denote the topological support of $\mu$.
\begin{itemize}
\item $\mu$ is called \emph{absolutely decaying (resp. decaying)} if there exist $C_1,\alpha > 0$ such that for all $\xx\in\Supp(\mu)$, $0 < \rho \leq 1$, $\beta > 0$, and $\LL\in\Hyp$, if $B = B(\xx,\rho) \subset U$ then
\begin{align}
\label{absolutelydecaying}
\mu\big(\thickopenvar\LL{\beta\rho}\cap B\big) &\leq C_1 \beta^\alpha\mu(B) & \text{(absolutely decaying)}
\end{align}
or
\begin{align}
\label{decaying}
\mu\big(\thickopenvar\LL{\beta \|d_\LL\|_{\mu,B}}\cap B\big) &\leq C_1 \beta^\alpha\mu(B) & \text{(decaying)},
\end{align}
respectively, where
\[
\|d_\LL\|_{\mu,B} := \sup\{\dist(\yy,\LL):\yy\in B\cap\Supp(\mu)\}.
\]
\item $\mu$ is called \emph{nonplanar} if $\mu(\LL) = 0$ for all $\LL\in\Hyp$. Note that every absolutely decaying measure is nonplanar. Moreover, the decaying and nonplanarity conditions can be combined notationally by using closed thickenings rather than open ones: a measure $\mu$ is decaying and nonplanar if and only if there exist $C_1,\alpha > 0$ such that for all $\xx\in\Supp(\mu)$, $0 < \rho \leq 1$, $\beta > 0$, and $\LL\in\Hyp$, if $B = B(\xx,\rho) \subset U$ then
\begin{align}
\label{decayingprime}
\mu\big(\thickvar\LL{\beta \|d_\LL\|_{\mu,B}}\cap B\big) &\leq C_1 \beta^\alpha\mu(B). & \text{(decaying and nonplanar)}
\end{align}
\item $\mu$ is called \emph{Federer} (or \emph{doubling}) if for some (equiv. for all) $K > 1$, there exists $C_2 > 0$ such that for all $\xx\in\Supp(\mu)$ and $0 < \rho\leq 1$, if $B(\xx,K\rho)\subset U$ then
\begin{equation}
\label{federer}
\mu\big(B(\xx,K\rho)\big) \leq C_2\mu\big(B(\xx,\rho)\big).
\end{equation}
\end{itemize}
If $\mu$ is Federer, decaying, and nonplanar, then $\mu$ is called \emph{friendly}; if $\mu$ is both absolutely decaying and Federer, then $\mu$ is called \emph{absolutely friendly}.\Footnote{As KLW put it, the word ``friendly'' is ``a somewhat fuzzy abbreviation of \emph{Federer, nonplanar, and decaying}''.} When the open set $U$ is not explicitly mentioned, we assume that it is all of $\R^d$; otherwise we say that $\mu$ is absolutely decaying, friendly, etc. ``relative to $U$''.
\end{definition}
\begin{definition}
\label{definitionQD}
Let $\mu$ be a measure on $\R^d$ and consider $\xx\in \borel\subset\R^d$. We will say that $\mu$ is \emph{quasi-decaying (resp. weakly quasi-decaying) at $\xx$ relative to $\borel$} if for all $\gamma > 0$, there exist $C_1,\alpha > 0$ such that for all $0 < \rho \leq 1$, $0 < \beta \leq \rho^\gamma$, and $\LL\in\Hyp$, if $B = B(\xx,\rho)$ then
\begin{align}
\label{QDwithE}
\mu\left(\thickvar{\LL}{\beta\rho}\cap B\cap \borel\right) &\leq C_1 \beta^\alpha \mu(B) & \text{(quasi-decaying)}
\end{align}
or
\begin{align}
\label{weakQDwithE}
\mu\left(\thickvar{\LL}{\beta\|\dist_\LL\|_{\mu,B}}\cap B\cap \borel\right) &\leq C_1 \beta^\alpha \mu(B) & \text{(weakly quasi-decaying)},
\end{align}
respectively. We will say that $\mu$ is \emph{(weakly) quasi-decaying relative to $\borel$} if for $\mu$-a.e. $\xx\in \borel$, $\mu$ is (weakly) quasi-decaying at $\xx$ relative to $\borel$. Finally, we will say that $\mu$ is \emph{(weakly) quasi-decaying} if there exists a sequence $(\borel_\Index)_\Index$ such that $\mu\left(\R^d\butnot\bigcup_\Index \borel_\Index\right) = 0$ and for each $\Index$, $\mu$ is (weakly) quasi-decaying relative to $\borel_\Index$.
\end{definition}

We also recall that the following implications hold:

\begin{center}
\begin{tabular}{|ccc|}
\hline
Absolutely friendly & \implies & Friendly\\
$\Downarrow$ & & $\Downarrow$\\
Quasi-decaying & \implies & Weakly quasi-decaying\\
\hline
\end{tabular}
\end{center}
Moreover, all four classes of measures are contained in the class of extremal measures. Now we are able to go farther, using examples from later in this paper as well as from the literature to show that these implications are all strict. See Figure \ref{figureexamples} for more details.

\begin{figure}
\scriptsize
\thispagestyle{plain}
\begin{tabular}{||c||c|c|c||}
\hline
\hline
& Absolutely friendly & \spc{Friendly but not \\ absolutely friendly} & Not friendly\\
\hline
\hline
QD & \spc{$\bullet$ Patterson--Sullivan measures \\ of convex-cocompact groups \\\\ $\bullet$ Equilibrium states of \\ finite IFSes and \\ hyperbolic rational functions} & \spc{$\bullet$ Patterson--Sullivan measures \\ of geometrically finite groups \\ which satisfy $k_{\min} < d$} & \spc{$\bullet$ Equilibrium states \\ of nonplanar infinite IFSes \\ and rational functions\\ (Appendix \ref{appendix})}\\
\hline
\spc{WQD$\butnot$QD} & Impossible & \spc{$\bullet$ Lebesgue measures of \\ nondegenerate manifolds} & \spc{$\bullet$ Conformal measures of \\ infinite IFSes which \\ have invariant spheres}\\
\hline
\spc{E$\butnot$WQD} & Impossible & Impossible & \spc{$\bullet$ Measures with finite \\ Lyapunov exponent and \\ zero entropy under \\ the Gauss map}\\
\hline
Not E & Impossible & Impossible & \spc{$\bullet$ Generic invariant measures of \\ hyperbolic toral endomorphisms \\\\ $\bullet$ Certain measures with \\ infinite Lyapunov exponent \\ under the Gauss map \\ \cite[Theorem 4.5]{FSU1}}\\
\hline
\hline
\end{tabular}
\caption{Some representative examples of dynamically defined measures compared on two axes: quasi-decay and friendliness. In the leftmost column we use the abbreviations QD = quasi-decaying, WQD = weakly quasi-decaying, and E = extremal. }
\label{figureexamples}
\end{figure}

\subsection{Ahlfors regularity vs. exact dimensionality}
One way of thinking about the difference between KLW's conditions and our conditions is by comparing this difference with the difference between the classes of \emph{Ahlfors regular} and \emph{exact dimensional} measures, both of which are well-studied in dynamics. We recall their definitions:

\begin{definition*}
A measure $\mu$ on $\R^d$ is called \emph{Ahlfors $\delta$-regular} if there exists $C > 0$ such that for every ball $B(\xx,\rho)$ with $\xx\in\Supp(\mu)$ and $0 < \rho \leq 1$.
\[
C^{-1} \rho^\delta \leq \mu\big(B(\xx,\rho)\big) \leq C \rho^\delta.
\]
The measure $\mu$ is called \emph{exact dimensional of dimension $\delta$} if for $\mu$-a.e. $\xx\in\R^d$,
\begin{equation}
\label{exactdim}
\lim_{\rho\searrow 0} \frac{\log\mu\big(B(\xx,\rho)\big)}{\log\rho} = \delta.
\end{equation}
\end{definition*}

The philosophical relations between Ahlfors regularity and exact dimensionality with absolute friendliness and quasi-decay, respectively, are:
\begin{equation}
\label{philosophical}
\begin{split}
\text{Ahlfors regular and ``nonplanar''} \;\; &\Rightarrow \;\; \text{Absolutely friendly}\\
\text{Exact dimensional and ``nonplanar''} \;\; &\Rightarrow \;\; \text{Quasi-decaying}
\end{split}
\end{equation}
Here ``nonplanar'' does not refer to nonplanarity as defined in Definition \ref{definitionabsolutelyfriendly}, but is rather something less precise (and stronger). This less precise definition should rule out examples like the Lebesgue measures of nondegenerate manifolds, since these are not quasi-decaying. One example of a ``sufficient condition'' for this imprecise notion of ``nonplanarity'' is simply the inequality $\delta > d - 1$, where $\delta$ is the dimension of the measure in question. In particular, in this context the relations \eqref{philosophical} are made precise by the following theorems:

\begin{theorem}[{\cite[Proposition 6.3]{KleinbockWeiss1}}; cf. \cite{PollingtonVelani,Urbanski3}]
\label{theoremahlforsregular}
If $\delta > d - 1$, then every Ahlfors $\delta$-regular measure on $\R^d$ is absolutely friendly.
\end{theorem}

\begin{theorem}[{\cite[Theorem 1.5]{DFSU_GE1}}]
\label{theoremexactdim}
If $\delta > d - 1$, then every exact dimensional measure on $\R^d$ of dimension $\delta$ is quasi-decaying.
\end{theorem}

\subsection{Examples of friendly and absolutely friendly measures}
\label{subsectionfriendlyexamples}
The two most canonical examples of friendly measures are the Lebesgue measure of a nondegenerate submanifold of $\R^d$ and the image of an absolutely friendly measure under a nondegenerate embedding \cite[Theorem 2.1]{KLW}. Neither of these measures are absolutely friendly; other than Lebesgue measure, the most canonical example of an absolutely friendly measure is the Hausdorff measure of the limit set of a finite irreducible iterated function system satisfying the open set condition. The following theorem is a slight refinement of KLW's original theorem regarding such measures. We note that the refinement is to weaken the irreducibility hypothesis; in \cite[Theorem 2.3]{KLW} the assumption that $K$ is not contained in any finite union of affine hyperplanes is required, but it turns out that this assumption is in fact equivalent to the more natural and formally weaker assumption that $K$ is not contained in any affine hyperplane.

\begin{theorem}[{\cite[Proposition 3.1]{BFS1}}, cf. {\cite[Theorem 2.3]{KLW}}]
\label{theoremsimilarityIFS}
Let $\{u_1,\ldots,u_m\}$ be a family of contracting similarities of $\R^d$ satisfying the open set condition (see Definition \ref{definitionCIFS}), and let $K$ be the limit set of this family (see Definition \ref{definitionlimitset}). If $K$ is not contained in an affine hyperplane, then $\mu = \scrH^\delta\given K$ is absolutely friendly, where $\delta = \HD(K)$.
\end{theorem}
Here and hereafter, $\HD(S)$ denotes the Hausdorff dimension of a set $S$, and $\scrH^s(S)$ denotes the $s$-dimensional Hausdorff measure of $S$.

Theorem \ref{theoremsimilarityIFS} was generalized by the fourth-named author as follows:

\begin{theorem}[{\cite[Corollary 1.6]{Urbanski}}\Footnote{Although this theorem includes the assumption $d\geq 2$, the case $d = 1$ is true due to Theorem \ref{theoremahlforsregular} and \cite[Theorem 3.14]{MauldinUrbanski1}.}]
\label{theoremconformalfinite}
Let $\{u_1,\ldots,u_m\}$ be a finite irreducible conformal iterated function system (CIFS) (see Definition \ref{definitionCIFS}) on $\R^d$ satisfying the strong open set condition, and let $K$ be the limit set of this family. Then $\mu = \scrH^\delta\given K$ is absolutely friendly, where $\delta = \HD(K)$.
\end{theorem}
\begin{remark} 
Although KLW write that Theorem \ref{theoremconformalfinite} provides a partial answer to a conjecture of theirs \cite[Conjecture 10.6 and ``Added in proof'' below]{KLW}, it actually provides a complete answer. Indeed, suppose that a CIFS satisfies the irreducibility assumption of KLW but not that of \cite{Urbanski}. Then the limit set $K$ of this CIFS is contained in a proper nondegenerate real-analytic submanifold $M\subset \R^d$, which if $d\geq 3$ is a sphere. Let $\phi:U\to M$ be a coordinate chart, where $U\subset \R^{d - 1}$. If $d\geq 3$, assume that $\phi$ is stereographic projection. Then $\phi^{-1}(K)$ is the limit set of a CIFS which satisfies the irreducibility assumption of \cite{Urbanski}. Thus by Theorem \ref{theoremconformalfinite}, $\HH^\delta\given \phi^{-1}(K)$ is absolutely friendly, and by \cite[Theorem 2.1(b)]{KLW}, 
\[
\HH^\delta\given K \asymp_\times \phi_*[\HH^\delta\given \phi^{-1}(K)]
\] 
is friendly and thus strongly extremal.
\end{remark}

\begin{remark} 
Note that the hypotheses of Theorems \ref{theoremsimilarityIFS} and \ref{theoremconformalfinite} imply (without the use of the irreducibility/nonplanarity hypothesis) that $\mu$ is Ahlfors $\delta$-regular \cite[Theorem 3.14]{MauldinUrbanski1}, so these theorems provide an example of the philosophical interpretation \eqref{philosophical} described above.
\end{remark}

\begin{remark}
Actually, the main theorem of \cite{Urbanski} concerned a class of measures more general than those of the form $\scrH^\delta\given K$: the equilibrium states of H\"older continuous potential functions. Such measures are not in general Federer unless a separation condition which is stronger than the open set condition is assumed. For example, if $u_a(x) = (x + a)/2$, $\phi_a(x) = \log(2^a/3)$ ($a = 0,1$, $x\in [0,1]$), and $\mu_\phi$ is the image under binary expansion of the Bernoulli measure on $\{0,1\}^\N$ corresponding to digit frequencies $\mathrm{freq}(0) = 1/3$, $\mathrm{freq}(1) = 2/3$; then $\mu_\phi$ is a equilibrium state of $\phi$, and it follows that for large values of $n$ the intervals $[1/2 - 1/2^n,1/2]$ and $[1/2,1/2 + 1/2^n]$ have $\mu_\phi$-measures which are not comparable. This technicality caused a minor error in the statement of \cite[Theorem 1.5]{Urbanski}, which we correct below.
In fact, due to this error, Theorem \ref{theoremconformalfinite} cannot be deduced as a corollary of Theorem \ref{theoremgibbsfinite}, but since the measure in Theorem \ref{theoremconformalfinite} is Federer by \cite[Theorem 3.14]{MauldinUrbanski1}, the proof of Theorem \ref{theoremgibbsfinite} actually proves Theorem \ref{theoremconformalfinite} as well. \label{footnoterem12}
The correct statement of \cite[Theorem 1.5]{Urbanski} should read as follows:
\begin{theorem}[{\cite[Theorem 1.5]{Urbanski}}]
\label{theoremgibbsfinite}
Let $\{u_1,\ldots,u_m\}$ be a finite irreducible CIFS on $\R^d$. Let $\phi:\{1,\ldots,m\}\to\R$ be a H\"older continuous potential function, and let $\mu_\phi$ be an equilibrium state of $\phi$\Footnote{Or more precisely, the image of an equilibrium state of $\phi$ under the coding map $\pi:\{1,\ldots,m\}^\N \to \R^d$.} (see Definition \ref{definitiongibbs}). If $\mu_\phi$ is Federer (e.g. if $\mu_\phi$ satisfies the strong separation condition, see Definition \ref{definitionCIFS}), then $\mu_\phi$ is absolutely friendly.
\end{theorem}
\end{remark}

All three of these theorems concern measures related to conformal iterated function systems. There is a close relationship between CIFSes and two other classes of conformal dynamical systems: Kleinian groups and rational functions. The relationship between these last two classes is sometimes known as ``Sullivan's dictionary'' and in \cite[p.4]{DSU_rigidity}, three of the authors proposed that the class of CIFSes should be added as a ``third column'' to Sullivan's dictionary. Thus, we should expect to find analogues of Theorems \ref{theoremsimilarityIFS}-\ref{theoremgibbsfinite} in the settings of Kleinian groups and rational functions. Indeed, the appropriate analogue in the setting of Kleinian groups was proven by B.~O.~Stratmann and the fourth-named author:

\begin{theorem}[\cite{StratmannUrbanski1}]
\label{theoremCCK}
Let $G$ be a convex-cocompact group of M\"obius transformations of $\R^d$ which does not preserve any generalized sphere (i.e. sphere or plane). Then the Patterson--Sullivan measure of $G$ (see \6\ref{subsectionpattersonsullivan}) is absolutely friendly.
\end{theorem}

If this theorem is the analogue of Theorem \ref{theoremconformalfinite}, then the analogue of Theorem \ref{theoremgibbsfinite} should concern the Patterson densities of H\"older continuous Gibbs cocycles as defined in \cite[p.3]{PPS}. In fact, it is not hard to see that the proof of \cite{StratmannUrbanski1} generalizes to this setting; we omit the details, as this theorem is not relevant to the overall goal of this paper.

A more difficult variation of Theorem \ref{theoremCCK} involves moving to a different row of Sullivan's dictionary (at least according to the table in \cite[p.4]{DSU_rigidity}): the setting of geometrically finite groups with parabolic points. The reason for this difficulty is that such measures are in general not absolutely friendly, because if $p$ is a parabolic point of a geometrically finite group $G$ with limit set $\Lambda$, and $\LL$ is an affine hyperplane containing the tangent plane of $\Lambda$ at $p$, then small neighborhoods of $\LL$ contain all the mass of small balls centered around $p$. This makes the Patterson--Sullivan measures of geometrically finite groups a good candidate for friendliness, since the quantity $\|d_\LL\|_{\mu,B}$ appearing in \eqref{decaying} is often much smaller than the radius of $B$. And indeed, we prove below:

\begin{theorem}[Proven in Section \ref{sectionGF}]
\label{theoremGFfriendly}
Let $G$ be a geometrically finite group of M\"obius transformations of $\R^d$ (see \6\ref{subsectionGFgroups}) which does not preserve any generalized sphere. Then the Patterson--Sullivan measure $\mu$ of $G$ is friendly. Moreover, $\mu$ is absolutely friendly if and only if every cusp of $G$ has maximal rank (see \6\ref{subsectionGFgroups}).
\end{theorem}

Moving to the third column of the dictionary, rational functions, it seems that the analogue of Theorems \ref{theoremsimilarityIFS}-\ref{theoremCCK} has not been explicitly stated in the literature before, but can be proven using the techniques of \cite{Urbanski} together with a rigidity result of W.~Bergweiler, A.~E.~Eremenko, and S.~J.~van~Strien \cite{BergweilerEremenko, EremenkoVanstrien}:

\begin{theorem}
\label{theoremHrational}
Let $T:\what\C\to\what\C$ be a hyperbolic (i.e. expansive on its Julia set) rational function, let $\phi:\what\C\to\R$ be a H\"older continuous potential function, and let $\mu_\phi$ be the corresponding equilibrium state. If $\Supp(\mu_\phi)$ is not contained in any generalized sphere, then $\mu_\phi$ is absolutely friendly.
\end{theorem}
\begin{proof}[Proof sketch]
Since $T$ is hyperbolic, there is some metric on the Julia set $J$ with respect to which $T$ is distance expanding. Thus by \cite[Theorem 4.5.2]{PrzytyckiUrbanski}, the map $T:J\to J$ has Markov partitions of arbitrarily small diameter. The inverse branches of $T$ restricted to the elements of such a partition forms a graph directed system in the sense of \cite{MauldinUrbanski2}, except that the cone condition is not necessarily satisfied. Now the equilibrium state $\mu_\phi$ is also an equilibrium state of this graph directed system with respect to the potential function $\phi$. The proof of \cite[Theorem 1.5]{Urbanski} shows that if $\mu_\phi$ is Federer and $\Supp(\mu_\phi)$ is not contained in any proper real-analytic submanifold of $\what\C$, then $\mu_\phi$ is absolutely friendly. (The cone condition is not used in the proof of \cite[Theorem 1.5]{Urbanski}, and the argument is easily extended from the realm of CIFSes to that of graph directed systems.)  But the fact that $\mu_\phi$ is an equilibrium state for $T$ implies that $\mu_\phi$ is Federer (e.g. by modifying the proof of \cite[Theorem A]{Rivera}). And if $\Supp(\mu_\phi)$ is not contained in any generalized sphere, then by \cite[Corollary 1 and Theorem 2]{EremenkoVanstrien} (see also \cite[Theorem 2]{BergweilerEremenko}), $\Supp(\mu_\phi)$ is not contained in any proper real-analytic submanifold of $\what\C$. This completes the proof.
\end{proof}

It seems that this is a more or less complete list of those fractal measures which, before the invention of the quasi-decay condition, were known to be extremal. The results are restricted to one or possibly two rows of Sullivan's dictionary, corresponding to extremely rigid hypotheses (finiteness of alphabet, convex-cocompactness, hyperbolicity) on the class of dynamical systems considered. In the authors' view, these examples represent more or less the full scope of the friendliness and absolute friendliness conditions: rigid hypotheses were required because the friendliness and absolute friendliness hypotheses are themselves very rigid. By contrast, the quasi-decaying condition is quite flexible and exhibits a much broader range of examples; let us now get a feel for what these examples are.

\subsection{Quasi-decay: examples and counterexamples}
\label{subsectionexamples}
The simplest examples of measures which are quasi-decaying but not necessarily Federer come from applying Theorem \ref{theoremexactdim} to measures invariant under one-dimensional dynamical systems. For these measures, exact dimensionality is known under fairly broad assumptions, and the dimension of exact dimensional measures can be computed directly:

\begin{theorem}[{\cite[Theorem 1]{Hofbauer}}]
\label{theoremhofbauer}
Let $T:[0,1]\to[0,1]$ be a piecewise monotonic transformation whose derivative has bounded $p$-variation for some $p > 0$. Let $\mu$ be a measure on $[0,1]$ which is ergodic and invariant with respect to $T$. Let $h(\mu)$ and $\chi(\mu)$ denote the entropy and Lyapunov exponent of $\mu$, respectively. If $\chi(\mu) > 0$, then $\mu$ is exact dimensional of dimension
\[
\delta(\mu) = \frac{h(\mu)}{\chi(\mu)}\cdot
\]
\end{theorem}

Note that if $h(\mu) > 0$, then Ruelle's inequality\Footnote{A proof valid in this context can be found in \cite[Theorem 7.1]{BarrioJimenez}.} implies that $\chi(\mu) > 0$, so the above result applies and gives $\delta_\mu > 0 = d - 1$, so $\mu$ is quasi-decaying.\Footnote{The inequality $\chi(\mu) < \infty$ follows from the hypothesis that $T'$ has bounded $p$-variation, which implies that $T'$ is bounded.} On the other hand, in general there is no reason to believe that measures satisfying the hypotheses of Theorem \ref{theoremhofbauer} will be either Federer or decaying.

We should also note that Theorem \ref{theoremhofbauer} provides examples of dynamical measures which are not quasi-decaying as well. Namely, if $\mu$ satisfies the hypotheses of Theorem \ref{theoremhofbauer} and $h(\mu) = 0$, then $\mu$ is exact dimensional of dimension $0$, and the following theorem shows that $\mu$ is not quasi-decaying:
\begin{theorem}[Proven in Section \ref{sectionmiscellaneous2}]
\label{theoremdimzero}
Any exact dimensional measure of dimension 0 is not quasi-decaying.
\end{theorem}
We note that the implication ($h(\mu) = 0$) \implies ($\mu$ exact dimensional of dimension 0) holds for higher-dimensional dynamical systems as well; see \cite[comment following Proposition 2]{BPS}.

Actually, it is not surprising that such measures are not quasi-decaying, because the class of such measures includes some measures which are not extremal. Specifically, if $T$ is the map $x\mapsto nx$ for some $n\geq 2$, then the following theorem proves the existence and genericity of $T$-invariant measures which are not extremal:

\begin{theorem}[Proven in Section \ref{sectionmiscellaneous2}]
\label{theoremliouville}
Let $T:X\to X$ be a hyperbolic toral endomorphism (cf. Definition \ref{definitionhyperbolictoral}), where $X = \R^d/\Z^d$. Let $\M_T(X)$ be the space of $T$-invariant probability measures on $X$. Then the set of measures which give full measure to the Liouville points is comeager in $\M_T(X)$.
\end{theorem}

In the one-dimensional case (i.e. $T(x) = nx$ mod 1 for some $n\geq 2$), combining with Theorems \ref{theoremhofbauer} and \ref{theoremexactdim} implies that the measures in Theorem \ref{theoremliouville} all have entropy zero with respect to $T$. Actually, this result, namely that generic invariant measures have entropy zero, holds more generally for all piecewise monotonic transformations of the interval \cite[Theorem 2(vi)]{Hofbauer2}, as well as for all Axiom A diffeomorphisms \cite[6th theorem on p.101]{Sigmund}.

There are numerous other classes of measures coming from dynamics which are known to be exact dimensional. A notable example is the theorem of Barreira, Pesin, and Schmeling \cite{BPS} to the effect that any measure ergodic, invariant, and hyperbolic with respect to a diffeomorphism is exact dimensional. Theorem \ref{theoremexactdim} applies directly to those measures whose dimension is sufficiently large, but the question still remains about those measures whose dimension is not large enough. We hope to return to this question, answering it at least for the class of expanding toral endomorphisms with distinct eigenvalues, and maybe even for possibly a much larger class. But for now, we turn our attention to measures coming from conformal dynamical systems.

\subsection{Main results: conformal examples of quasi-decaying measures}
We now come to the main results of this paper, namely extensions of the results of \6\ref{subsectionfriendlyexamples} to much broader classes of conformal dynamical systems, replacing friendliness and absolute friendliness by the quasi-decay condition. Our first theorem and its corollary generalize Theorems \ref{theoremsimilarityIFS}-\ref{theoremgibbsfinite}. It will be proven via a more general theorem which will allow us to more easily deduce that certain measures are quasi-decaying in later papers in this series.

\begin{theorem}[Proven in Section \ref{sectiongibbs}]
\label{theoremgibbs}
Fix $d\in\N$, and let $(u_a)_{a\in \alphabet}$ be an irreducible CIFS on $\R^d$. Let $\phi: \alphabet^\N\to\R$ be a summable locally H\"older continuous potential function, let $\mu_\phi$ be an equilibrium state of $\phi$, and let $\pi:\alphabet^\N\to\R^d$ be the coding map. Suppose that the Lyapunov exponent
\begin{equation}
\label{lyapunovmu}
\chi_{\mu_\phi} := \int \log(1/|u_{\omega_1}'(\pi\circ\sigma(\omega))|) \;\dee\mu_\phi(\omega)
\end{equation}
is finite. Then $\pi_*[\mu_\phi]$ is quasi-decaying.
\end{theorem}

Letting $\phi(\omega) = \delta\log(|u_{\omega_1}'(\pi\circ\sigma(\omega))|)$, where $\delta = \HD(K)$, yields the following corollary:

\begin{corollary}
\label{corollaryconformal}
Fix $d\in\N$, and let $(u_a)_{a\in \alphabet}$ be a regular (cf. \cite[Theorem 4.2.9]{MauldinUrbanski2}) irreducible CIFS on $\R^d$. Let $\mu$ be the conformal measure of $(u_a)_{a\in \alphabet}$, and let $\pi:\alphabet^\N\to\R^d$ be the coding map. If the Lyapunov exponent of $\mu$ is finite, then $\pi_*[\mu]$ is quasi-decaying.
\end{corollary}

The improvements on Theorems \ref{theoremconformalfinite} and \ref{theoremgibbsfinite} are twofold:
\begin{itemize}
\item The CIFS can be infinite, as long as the Lyapunov exponent is finite.
\item The open set condition is no longer needed.
\end{itemize}
Both of these improvements are quite significant. Without the open set condition it is hard to even calculate the dimension of the limit set of a CIFS; see e.g. \cite{Barany, PollicottSimon, SimonSolomyak}. Also, the geometry of infinite alphabet CIFSes can be much wilder than the geometry of finite CIFSes, whose limit sets are always Ahlfors regular (cf. \cite[Theorem 3.14]{MauldinUrbanski1} versus \cite[Lemma 4.12 - Theorem 4.16]{MauldinUrbanski1}). By contrast, the finite Lyapunov exponent assumption which replaces it is quite weak; for example, in the case of conformal measures it is implied by strong regularity (cf. \cite[Definition 4.3.1]{MauldinUrbanski2}). It is also a necessary assumption, as demonstrated by certain IFSes related to continued fractions \cite[Theorem 4.5]{FSU1}.

A connection between the finite Lyapunov exponent condition and extremality also appeared in an earlier paper of three of the authors:

\begin{theorem}[{\cite[Theorem 2.1]{FSU1}}]
\label{theoremFSU}
If $\mu$ is a probability measure on $[0,1]\butnot\Q$ invariant with finite Lyapunov exponent under the Gauss map, then $\mu$ is extremal.
\end{theorem}

This theorem is neither more nor less general than Theorem \ref{theoremgibbs}. It is obviously not more general, since it only applies to the Gauss map, and the conclusion that $\mu$ is extremal is weaker than the conclusion of Theorem \ref{theoremgibbs} which states that the relevant measure is quasi-decaying. But it is also not less general, since it applies to all invariant measures with finite Lyapunov exponent, and not only those which are equilibrium states of summable H\"older families of potential functions. The existence of invariant measures with finite Lyapunov exponent which are not quasi-decaying can be seen from Theorem \ref{theoremdimzero}, since the image of an entropy zero shift-invariant measure on $\{1,2\}^\N$ under the continued fractions / Gauss IFS coding map is such a measure. So Theorem \ref{theoremFSU} would become false if ``extremal'' were replaced by ``quasi-decaying''. It would also become false if ``Gauss map'' were replaced by ``shift map of a CIFS'', due to Theorem \ref{theoremliouville}, which produces non-extremal measures invariant under the map $x\mapsto 2x$, which is the shift map of the binary IFS. This indicates that the phenomenon captured in Theorem \ref{theoremFSU} is really a ``number-theoretic'' phenomenon arising directly from the connection between the Gauss map and Diophantine approximation, rather than from some geometric intermediary.

Moving on to our next conformal setting, we consider the Patterson--Sullivan measures of geometrically finite Kleinian groups:

\begin{theorem}[Proven in Section \ref{sectionGF}]
\label{theoremGFQD}
Let $G$ be a geometrically finite group of M\"obius transformations of $\R^d$ which does not preserve any generalized sphere. Then the Patterson--Sullivan measure $\mu$ of $G$ is quasi-decaying.
\end{theorem}

At first, this theorem may not seem to give any new information beyond that provided in Theorem \ref{theoremGFfriendly}: that Patterson--Sullivan measures of irreducible geometrically finite groups are extremal. But actually, the quasi-decay condition implies more Diophantine properties than friendliness does, via \cite[Theorem 1.3]{DFSU_GE1}, which shows that the image of a quasi-decaying measure under a nondegenerate embedding is still extremal. Such a theorem cannot apply to friendly measures, since the Lebesgue measure of a nondegenerate manifold is friendly, but its image under a nondegenerate embedding may be contained in a rational hyperplane and therefore non-extremal. But Theorem \ref{theoremGFQD} shows that such a fate cannot befall a Patterson--Sullivan measure. Two examples give this observation special significance: first of all, we may wish to consider the Patterson--Sullivan measure of a group of M\"obius transformations of $\R^d$ which preserves the unit sphere but not any smaller generalized sphere; this measure is absolutely continuous to the image under a nondegenerate embedding of a geometrically finite group acting irreducibly on $\R^{d - 1}$, so Theorem \ref{theoremGFQD} and \cite[Theorem 1.3]{DFSU_GE1} imply that the measure is extremal. Second of all, if we consider $\R^d$ as isomorphic with the space of $M\times N$ matrices for some $M,N\geq 2$ such that $MN = d$, then Theorem \ref{theoremGFQD} and \cite[Corollary 1.8]{DFSU_GE1} imply that the Patterson--Sullivan measure gives zero measure to the set of very well approximable $M\times N$ matrices, but such a conclusion cannot be deduced from Theorem \ref{theoremGFfriendly}.

Also, treating Theorem \ref{theoremGFfriendly} as though it ``came first'' is a bit silly in the sense that both theorems use the same main lemma (Lemma \ref{lemmaglobaldecay}), and the argument from that lemma to Theorem \ref{theoremGFQD} is significantly easier than the argument from that lemma to Theorem \ref{theoremGFfriendly}. In any case, it is interesting to have a natural example of a measure which is both quasi-decaying and friendly but not absolutely friendly.

Our last theorem relating quasi-decay to conformal dynamics is in the setting of rational functions. Its proof uses results of M.~Szostakiewicz, A.~Zdunik, and the fourth-named author that utilize ``fine inducing'', as well as the rigidity result of W.~Bergweiler, A.~E.~Eremenko, and S.~J.~van Strien mentioned earlier. We recall that if $T:X\to X$ is a dynamical system, then a potential function $\phi:X\to\R$ is called \emph{hyperbolic} if there exists $n\in\N$ such that $\sup(S_n \phi) < P(T^n,S_n \phi)$, where $S_n\phi \df \sum_{k=0}^{n-1} \phi\circ T^k$ and $P(T,\phi)$ is the pressure of $\phi$ with respect to $T$.
(The notion of hyperbolic potentials was introduced in \cite{InoquioRivera} and generalizes the notion of \emph{pressure gap} corresponding to $n=1$, which was introduced in \cite{DenkerUrbanski}. Note that $\phi$ is a hyperbolic potential with respect to $T$ if and only if $S_n\phi$ has pressure gap with respect to  $T^n$ for some $n\ge 1$.)

\begin{theorem}
\label{theoremrational}
Let $T:\what\C\to\what\C$ be a rational function, let $\phi:\what\C\to\R$ be a H\"older continuous hyperbolic potential function, and let $\mu_\phi$ be the equilibrium state of $(T,\phi)$ (for existence and uniqueness see \cite{DenkerUrbanski}). If the Julia set of $T$ is not contained in a generalized sphere, then $\mu$ is quasi-decaying.
\end{theorem}
\begin{proof}
Let $(u_a)_{a\in \alphabet}$ be the CIFS described on \cite[p.20]{SUZ1}. By \cite[Lemma 20]{SUZ1}, there exists a locally H\"older continuous potential function $\wbar\phi$ such that $\mu_\phi$ is (the image under the coding map of) the equilibrium state of $\wbar\phi$ with respect to $(u_a)_{a\in \alphabet}$. By \cite[(22)]{SUZ1}, the function $\wbar\phi$ is summable, and by \cite[Proposition 23]{SUZ1}, the measure $\mu_\phi$ has finite Lyapunov exponent with respect to $(u_a)_{a\in \alphabet}$. So by Theorem \ref{theoremgibbs}, if $(u_a)_{a\in \alphabet}$ acts irreducibly then $\mu_\phi$ is quasi-decaying. But if $(u_a)_{a\in \alphabet}$ acts reducibly, then $\mu_\phi$ is supported on a proper real-analytic submanifold of $\what\C$, so the Julia set $J = \Supp(\mu_\phi)$ is contained in a proper real-analytic submanifold of $\what\C$. So by \cite[Corollary 1 and Theorem 2]{EremenkoVanstrien} (see also \cite[Theorem 2]{BergweilerEremenko}), $J$ is contained in a generalized sphere.
\end{proof}

Comparing this theorem to Theorem \ref{theoremHrational}, we see that we have replaced the hypothesis that the rational function $T$ is hyperbolic with the hypothesis that the potential function $\phi$ is hyperbolic. Although these conditions sound superficially similar due to the prolific use of the word ``hyperbolic'' in mathematical definitions, the latter is actually much weaker than the former, which essentially means that both critical points and parabolic points are irrelevant to the dynamics. By contrast, the hypothesis that $\phi$ is hyperbolic does not actually place any restriction on the function $T$ (given any rational function $T$, the function $\phi \equiv 0$, or more generally any function $\phi$ satisfying $\sup(\phi) - \inf(\phi) < \log\deg(T)$, is hyperbolic), and in fact only fails when there is an equilibrium state whose entropy and Lyapunov exponent are both zero \cite{InoquioRivera}. In particular, if $T$ is a Topological Collet--Eckmann (TCE) map then every H\"older continuous potential function $\phi$ is hyperbolic \cite[Corollary 1.1]{InoquioRivera}.

\bigskip
{\bf Outline of the paper.} In Section \ref{sectionmiscellaneous2} we prove Theorems \ref{theoremdimzero} and \ref{theoremliouville}, thus providing some examples of measures which are not quasi-decaying. In Section \ref{sectionGF} we prove Theorems \ref{theoremGFfriendly} and \ref{theoremGFQD}, thus showing that Patterson--Sullivan measures of geometrically finite groups are both friendly and quasi-decaying. In Section \ref{sectiongibbs} we prove Theorem \ref{theoremgibbs}, thus showing that the equilibrium states of conformal iterated function systems are quasi-decaying. Finally, we provide some examples of dynamically defined measures that are quasi-decaying but not friendly in Appendix \ref{appendix}.

\begin{remark*}
We refer to \cite[\63]{DFSU_GE1} for some preliminary results which will be used in the proofs, in particular the ``quasi-Federer lemma'' \cite[Lemma 3.2]{DFSU_GE1} and the ``Lebesgue differentiation theorem'' \cite[Theorem 3.6]{DFSU_GE1}.
\end{remark*}

\draftnewpage
\section{Proofs of Theorems \ref{theoremdimzero} and \ref{theoremliouville}}
\label{sectionmiscellaneous2}

\begin{proof}[Proof of Theorem \ref{theoremdimzero}]
Let $\mu$ be exact dimensional of dimension 0, and let $\borel\subset\R^d$ be a set such that $\mu(\borel) > 0$. Then for a $\mu$-typical point $\xx\in \borel$, by the exact dimensionality of $\mu$ we have
\[
\lim_{\rho\searrow 0} \frac{\log\mu\big(B(\xx,\rho)\big)}{\log(\rho)} = 0
\]
while by the Lebesgue differentiation theorem \cite[Theorem 3.6]{DFSU_GE1}, \cite[Theorem 9.1]{Simmons1}, we have
\[
\lim_{\rho\searrow 0} \frac{\mu\big(B(\xx,\rho)\cap \borel\big)}{\mu\big(B(\xx,\rho)\big)} = 1.
\]
It follows that
\[
\lim_{\rho\searrow 0}\log\frac{\mu\big(B(\xx,\rho^2)\cap \borel\big)}{\rho^\alpha \mu\big(B(\xx,\rho)\big)} = \infty \text{ for all $\alpha > 0$},
\]
which implies that $\mu$ is not quasi-decaying at $\xx$ relative to $\borel$. Since $\xx$ was a $\mu$-typical point, $\mu$ is not quasi-decaying relative to $\borel$; since $\borel$ was arbitrary subject to $\mu(\borel) > 0$, $\mu$ is not quasi-decaying.
\end{proof}

Before proving Theorem \ref{theoremliouville}, we recall the definitions of some terms used in its statement.

\begin{definition}
\label{definitionhyperbolictoral}
Let $X = \R^d/\Z^d$. A \emph{toral endomorphism} of $X$ is a map $T:X\to X$ of the form $T([\xx]) = [\bfM\xx]$, where $\bfM$ is a $d\times d$ matrix with integer entries. Here $[\xx]$ denotes the image of a point $\xx\in\R^d$ under the quotient map $\R^d\to \R^d/\Z^d$. The endomorphism $T$ is called \emph{hyperbolic} if the eigenvalues of $\bfM$ all have modulus $\neq 1$.
\end{definition}

\begin{proof}[Proof of Theorem \ref{theoremliouville}]
For each $n\in\N$, let
\[
U_n = \bigcup_{\substack{\pp/q\in\Q \\ q\geq n}} B\left(\frac\pp q,\frac1{q^n}\right),
\]
and let $\UU_n$ be the set of all measures $\mu\in\M_T(X)$ such that $\mu(U_n) > 1 - 2^{-n}$. The sets $U_n$ and $\UU_n$ are both open. The set $G := \bigcap_n U_n$ is the set of Liouville points, i.e. points with infinite exponent of irrationality. Thus since every measure in $\GG := \bigcap_n \UU_n$ gives full measure to $G$, it follows that no measure in $\GG$ is extremal.

To complete the proof, we need to show that $\GG$ is dense in $\M_T(X)$. Since $\GG$ is convex, it suffices to show that the closure of $\GG$ contains all ergodic measures in $\M_T(X)$. Indeed, let $\mu\in\M_T(X)$ be an ergodic measure, and let $\xx\in X$ be a $\mu$-random point. Then by the ergodic theorem,
\[
\mu_N := \frac1N \sum_{i = 0}^{N - 1} \delta_{T^i(\xx)} \to \mu,
\]
where $\delta$ denotes the Dirac delta. Fix $\epsilon > 0$. Since $T$ is hyperbolic, it has the specification property \cite[(2.10) Theorem]{Bowen3}. (Although the result of \cite{Bowen3} as stated only covers the case of invertible transformations, one may verify the specification property for hyperbolic toral endomorphisms by direct calculation.) 
Thus there exists $m = m_\epsilon\in\N$ such that for all $N\geq m$, there exists $\yy = \yy_{N,\epsilon}\in X$ such that $T^N(\yy) = \yy$ and
\[
\dist(T^i(\xx),T^i(\yy)) \leq \epsilon \all i = 0,\ldots,N - m - 1.
\]
Next fix $N\geq m$ and let
\[
\nu_{N,\epsilon} = \frac1N \sum_{i = 0}^{N - 1} \delta_{T^i(\yy)} \in \M_T(X).
\]
If $\bfM$ is the integer matrix representing $T$, then $(\bfM^N - I)$ is a nonsingular integer matrix and $(\bfM^N - I)\yy\in\Z^d$, so $\yy\in\Q^d$. Thus $\Supp(\nu_{N,\epsilon})\subset \Q^d\subset G$, so $\nu_{N,\epsilon}\in\GG$.

Let $\dist$ be the co-Lipschitz distance on $\M_T(X)$, i.e.
\[
\dist(\mu,\nu) = \sup_f \left|\int f\;\dee\mu - \int f\;\dee\nu\right|,
\]
where the supremum is taken over all $1$-Lipschitz functions $f:X\to [-1,1]$. Then
\begin{align*}
\dist(\nu_{N,\epsilon},\mu) &\leq \dist(\nu_{N,\epsilon},\mu_N) + \dist(\mu_N,\mu)\\
&\leq \dist(\mu_N,\mu) + \frac1N \sum_{i = 0}^{N - 1} \dist\big(T^i(\xx),T^i(\yy_{N,\epsilon})\big)\\
&\leq \dist(\mu_N,\mu) + \frac1N \left[\sum_{i = 0}^{N - m - 1} \epsilon + \sum_{i = N - m}^{N - 1} 1\right]\\
&\leq \dist(\mu_N,\mu) + \epsilon + \frac{m_\epsilon}N \tendsto{N\to\infty} \epsilon \tendsto{\epsilon\to 0} 0.
\end{align*}
Thus $\mu$ is in the closure of the set $\{\nu_{N,\epsilon} : N\geq m, \epsilon > 0\} \subset\GG$.
\end{proof}

\draftnewpage
\section{Patterson-Sullivan measures}
\label{sectionGF}

In this section we will prove Theorems \ref{theoremGFQD} and \ref{theoremGFfriendly}, namely that the Patterson--Sullivan measure of any irreducible geometrically finite Kleinian group is both quasi-decaying and friendly, but is absolutely friendly if and only if all cusps have full rank.

\subsection{Conformal and hyperbolic geometry}
We recall some preliminaries from conformal and hyperbolic geometry. Throughout this section, $\H = \H^{d + 1} = \{(x_0,\ldots,x_d)\in\R^{d + 1} : x_0 > 0\}$ and $\B = \B^{d + 1} = \{\xx\in\R^{d + 1} : \|\xx\| < 1\}$ will denote the upper half-space and Poincar\'e ball models of hyperbolic geometry, respectively. Their boundaries are $\del\H = \R^d\cup\{\infty\}$ and $\del\B = S^d \subset\R^{d + 1}$, respectively. Isometries of $\H$ and $\B$ correspond to conformal isomorphisms of their boundaries. In particular, $\H$ and $\B$ are isometric via the \emph{Cayley transform} $\iota:\H\to\B$ defined by\Footnote{Here $\ee_i$ denotes the $i$th basis vector of $\R^{d+1}$.}
\[
\iota(\xx) = 2\frac{\xx + \ee_0\;}{\|\xx + \ee_0\|^2} - \ee_0
\]
and this induces (via the same formula) a conformal isomorphism between $\del\H$ and $\del\B$, known as \emph{stereographic projection}. The \emph{spherical metric} on $\del\H$ is the pullback of the Euclidean metric on $\del\B$ under stereographic projection. We will denote it by $\Dist_\sph$, while denoting the Euclidean metric on $\del\H$ by $\Dist_\euc$. By contrast, we will denote the hyperbolic metric on $\H$ by $\dist_\hyp$, defined by the formula
\[
\cosh\dist_\hyp(\xx,\yy) = 1 + \frac{\|\yy - \xx\|^2}{2 x_0 y_0}
\]
(cf. \cite[(2.5.3)]{DSU}). We will also use the subscripts $\sph$, $\euc$, and $\hyp$ in referring to constructions in which these metrics are used, e.g. $B_\hyp(\xx,\rho)$, $\NN_\euc(\LL,\beta\sigma)$, $|g'(\xx)|_\sph$, etc.

A \emph{generalized sphere} in $\del\H$ is the pullback under stereographic projection of a set of the form $\del\B\cap\AA$, where $\AA$ is an affine subspace of $\R^{d + 1}$. Equivalently, a generalized sphere is either an affine subspace of $\del\H$ (including $\infty$ by convention) or the intersection of such a subspace with a sphere in $\del\H$. The collection of all generalized spheres of $\del\H$ of codimension 1 (i.e. dimension $d - 1$) will be denoted $\Sph$ (to contrast with $\Hyp$, the collection of all hyperplanes). The collection $\Sph$ is preserved under conformal transformations of $\del\H$.

We denote the set of isometries of $\H$ by $\Isom(\H)$. The following well-known results will be used frequently in the sequel:
\begin{theorem}[Geometric mean value theorem]
\label{theoremGMVT}
Fix $g\in \Isom(\H)$. Then for all $\xi,\eta\in\del\H$, the formula
\begin{equation}
\label{GMVT}
\frac{\Dist(g(\xi),g(\eta))}{\Dist(\xi,\eta)} = \big(|g'(\xi)|\cdot |g'(\eta)|\big)^{1/2}
\end{equation}
holds both for the Euclidean metric and for the spherical metric.
\end{theorem}

\begin{theorem}[Bounded distortion principle]
\label{theoremboundeddistortion}
Fix $g\in \Isom(\H)$ and let
\[
\|g\| = \dist_\hyp(\zero,g(\zero)),
\]
where $\zero = \ee_0\in\H$. For all $\xi\in\del\H$ we have
\[
|g'(\xi)|_\sph \leq e^{\|g\|}.
\]
Moreover, if we fix $\eta\in\del\H$ and let
\begin{align*}
d &= \dist_\hyp(g^{-1}(\zero),\geo\zero\eta)\\
\rho &= \Dist_\sph(\xi,\eta),
\end{align*}
then
\[
|g'(\xi)| \gtrsim_\times e^{\|g\| - d} / (1\vee (e^{\|g\|}\rho)).
\]
Here $\geo\zero\xi$ denotes the geodesic ray connecting $\zero$ and $\xi$.
\end{theorem}

\subsection{Geometrically finite groups}
\label{subsectionGFgroups}
Fix $G\leq \Isom(\H)$. The \emph{limit set} of $G$, denoted $\Lambda$ is the collection of accumulation points in $\del\H$ of the set $G(\zero)$. The \emph{convex hull} of $\Lambda$, denoted $\CC_G$, is the smallest convex subset of $\H$ whose closure contains $\Lambda$, and is empty if $\Lambda$ is a singleton. A \emph{horoball} is a set of the form
\[
H(\xi,t) = \{x\in\H : \busemann_\xi(\zero,x) > t\},
\]
where $\busemann_\xi$ denotes the \emph{Busemann function}
\[
\busemann_\xi(x,y) = \lim_{z\to\xi} [\dist_\hyp(z,x) - \dist_\hyp(z,y)].
\]
The point $\xi\in\del\H$ is said to be the \emph{center} of the horoball $H(\xi,t)$.

Next, $G$ is said to be \emph{geometrically finite} if there exist $\sigma > 0$ and a disjoint $G$-invariant collection of horoballs $\scrH$ such that
\begin{equation}
\label{GFdef}
\CC_G \subset G\big(B_\hyp(\zero,\sigma)\big) \cup \bigcup(\scrH).
\end{equation}
In this scenario, the collection $\scrH$ can be written in the form
\[
\scrH = \{H_\eta : \eta\in P\}
\]
where $P$ is the set of parabolic fixed points of $G$ and for each $\eta\in P$, $H_\eta$ is a horoball centered at $\eta$. We will also need the following well-known results, in which $G$ denotes a geometrically finite group:

\begin{theorem}[Cusp finiteness theorem]
\label{theoremcuspfiniteness}
There exists a finite set $P_0$ such that $P = G(P_0)$.
\end{theorem}

\begin{theorem}[Top representation theorem]
\label{theoremtoprepresentation}
For all $H\in\scrH$, there exist $\eta\in P_0$ and $g\in G$ such that $H = g(H_\eta)$ and
\[
\dist(g(\zero),\geo\zero x) \asymp_\plus 0 \all x\in H.
\]
\end{theorem}

\begin{theorem}[Boundedness of parabolic points]
\label{theoremBPP}
For all $\eta\in P$, there exists a compact set $\DD_\eta\subset\Lambda$ not containing $\eta$ such that
\[
\Lambda\butnot\{\eta\} = G_\eta(\DD_\eta),
\]
where $G_\eta$ is the stabilizer of $\eta$ in $G$.
\end{theorem}

\begin{theorem}[Translation planes]
\label{theoremtranslationplane}
For all $\eta\in P$, there exists a generalized sphere $\PP_\eta\subset\del\H$ containing $\eta$ such that $G_\eta$ preserves $\PP_\eta$ and acts cocompactly on it.
\end{theorem}

The dimension of $\PP_\eta$ is called the \emph{rank} of $\eta$, and is denoted $k_\eta$.

\begin{theorem}[Inequality between rank and Poincar\'e exponent]
\label{theorembeardon}
For all $\eta\in P$, we have $\delta > k_\eta/2$, where $\delta$ is the \underline{Poincar\'e exponent}
\begin{equation}
\label{poincare}
\delta = \inf\left\{s\geq 0 : \sum_{g\in G} e^{-s\|g\|} < \infty\right\}.
\end{equation}
\end{theorem}

\subsection{Patterson--Sullivan measures}
\label{subsectionpattersonsullivan}
The \emph{Patterson--Sullivan measure} of $G$, denoted $\mu$, is the measure on $\Lambda$, unique up to a multiplicative factor, such that for all $g\in G$ and $E\subset\del\H$,
\begin{equation}
\label{pattersonsullivan}
\mu\big(g(E)\big) = \int_E |g'(\xi)|^\delta \;\dee\mu(\xi),
\end{equation}
where $\delta$ is as in \eqref{poincare}. The existence of such a measure is due to S.~J.~Patterson and D.~P.~Sullivan \cite{Patterson2, Sullivan_density_at_infinity}, and its uniqueness in the case of geometrically finite groups is due to Sullivan \cite{Sullivan_entropy}. The support of the Patterson--Sullivan measure is the entire limit set, i.e. $\Supp(\mu) = \Lambda$. Note that the Patterson--Sullivan measure is dependent on the choice of metric ($\Dist_\sph$ or $\Dist_\euc$), so there are really two Patterson--Sullivan measures $\mu_\sph$ and $\mu_\euc$, related by the formula
\[
\frac{\dee\mu_\sph}{\dee\mu_\euc}(\xi) = \left(\lim_{\eta\to\xi}\frac{\Dist_\sph(\xi,\eta)}{\Dist_\euc(\xi,\eta)}\right)^\delta.
\]
Other than the transformation equation \eqref{pattersonsullivan}, the main facts we will need about the Patterson--Sullivan measure $\mu$ are:
\begin{itemize}
\item $\mu$ is Federer; this follows from the global measure formula \cite[\67]{Sullivan_entropy}, \cite[Theorem 2]{StratmannVelani}.
\item The following corollary of the logarithm law \cite[Theorem 4]{StratmannVelani}:  for $\mu$-a.e. $\xi\in \Lambda$, we have
\begin{equation}
\label{logarithmlaw}
\lim_{t\to\infty} \frac1t \dist_\hyp(\xi_t,G(\zero)) = 0,
\end{equation}
where for each $t\geq 0$, $\xi_t$ denotes the unique point on the geodesic ray connecting $\zero$ and $\xi$ such that $\dist_\hyp(\zero,\xi_t) = t$.
\end{itemize}

\subsection{From global decay to quasi-decay}
A group $G\leq\Isom(\H)$ is said to be \emph{irreducible} if $G$ does not preserve any generalized sphere strictly contained in $\del\H$. In the remainder of this section, $G$ will denote an irreducible geometrically finite group, $\mu$ its Patterson--Sullivan measure, etc. We will use the spherical metric as the default metric; in particular $\mu$ denotes the Patterson--Sullivan measure with respect to the spherical metric. The key lemma which we will use to prove both quasi-decay and friendliness is the following:

\begin{lemma}[``Global decay'']
\label{lemmaglobaldecay}
There exists $\alpha > 0$ such that for all $\beta > 0$ and $\LL\in\Sph$, we have
\begin{equation}
\label{globaldecay}
\mu\big(\thickvar\LL\beta\big) \lesssim_\times \beta^\alpha.
\end{equation}
\end{lemma}

Before proving this lemma, we use it to prove that $\mu$ is quasi-decaying:

\begin{proof}[Proof of Theorem \ref{theoremGFQD} assuming Lemma \ref{lemmaglobaldecay}]
Fix $\xi\in \Lambda\butnot\{\infty\}$ satisfying \eqref{logarithmlaw}, and we will show that $\mu$ is quasi-decaying at $\xi$ relative to $\R^d$. Since $\Dist_\euc \asymp_\times \Dist_\sph$ in a small neighborhood of $\xi$, it is enough to show that, quantified appropriately, the equation \eqref{QDwithE} is satisfied when the metric is taken to be the spherical metric. Fix $\gamma > 0$, $0 < \rho\leq 1$, $0 < \beta \leq \rho^\gamma$, and $\LL\in\Hyp$. Replace $\LL$ by $\LL \cup \{\infty\}$, so that $\LL\in\Sph$.

Let $t = -\log(\rho)$, and fix $g\in G$ such that $d = \dist(\xi_t,g(\zero)) = \dist(\xi_t,G(\zero))$. Then by Theorem \ref{theoremboundeddistortion},
\begin{equation}
\label{gprimebounds}
e^{t - 2d} \lesssim_\times |(g^{-1})'| \leq e^{t + d} \text{ on $B(\xi,\rho)$},
\end{equation}
so
\begin{align*}
\mu\big(\thickvar\LL{\beta\rho}\cap B(\xi,\rho)\big) &\lesssim_\times e^{-\delta(t - 2d)}\mu\big(g^{-1}(\thickvar\LL{\beta\rho})\big) \by{\eqref{pattersonsullivan}}\\
&\leq_\pt e^{-\delta(t - 2d)} \mu\big(\thickvar{g^{-1}(\LL)}{e^{t + d}\beta \rho}\big) \by{Theorem \ref{theoremGMVT}}\\
&\lesssim_\times e^{-\delta(t - 2d)} (e^{(t + d)} \beta \rho)^\alpha \by{Lemma \ref{lemmaglobaldecay}}\\
&=_\pt \beta^\alpha \rho^\delta e^{(\alpha + 2\delta)d}\\
&\lesssim_\times \beta^\alpha \rho^{\delta - \alpha\gamma/4} \by{\eqref{logarithmlaw}}\\
&\leq_\pt \beta^{\alpha/2} \rho^{\delta + \alpha\gamma/4}. \since{$\beta \leq \rho^\gamma$}
\end{align*}
A similar argument can be used to show that
\begin{equation}
\label{GFexactdim}
\lim_{\rho\searrow 0} \frac{\log\mu\big(B(\xi,\rho)\big)}{\log(\rho)} = \delta.
\end{equation}
(Alternatively, \eqref{GFexactdim} follows from \eqref{logarithmlaw} together with the global measure formula \cite[Theorem 2]{StratmannVelani}, although this is overkill.) Thus
\[
\rho^{\delta + \alpha\gamma/4} \lesssim_\times \mu\big(B(\xi,\rho)\big),
\]
which completes the proof.
\end{proof}

\begin{remark}
The above proof is valid for any group $G\leq\Isom(\H)$ (not necessarily geometrically finite) whose Patterson--Sullivan measure satisfies \eqref{logarithmlaw} and Lemma \ref{lemmaglobaldecay}. It would be interesting to see whether there are geometrically infinite examples of groups with these properties.
\end{remark}

\subsection{Proof of global decay}
We proceed to prove Lemma \ref{lemmaglobaldecay} via a series of reductions, finally reducing the question to one purely about a certain smooth measure.

\begin{lemma}
\label{lemmapreglobaldecay}
There exist $0 < \epsilon,\lambda < 1$ such that for all  $0 < \beta \leq 1$ and $\LL\in\Sph$, we have
\begin{equation}
\label{preglobaldecay}
\mu(\thickvar\LL{\epsilon\beta}) \leq \lambda \mu(\thickvar\LL{\beta}).
\end{equation}
\end{lemma}
\begin{proof}[Proof of Lemma \ref{lemmaglobaldecay} from Lemma \ref{lemmapreglobaldecay}]
Fix $n\in\N$. By applying \eqref{preglobaldecay} with $\beta = 1,\epsilon,\ldots,\epsilon^{n - 1}$, we have
\[
\mu\big(\thickvar\LL{\epsilon^n}\big) \leq \lambda^n \mu(\del\H) = \lambda^n.
\]
Thus \eqref{globaldecay} holds with $\alpha = \log_\epsilon(\lambda)$.
\end{proof}

\draftnewpage
\begin{lemma}
\label{lemmaexistsrho1}
For all $\rho_0 > 0$, there exists $\epsilon > 0$ with the following property: for all $0 < \beta \leq 1$, $\LL\in\Sph$, and $\xi\in \Lambda\cap\thickvar\LL{\epsilon\beta}$, there exists $0 < \rho \leq \rho_0$ such that
\begin{equation}
\label{existsrho1}
\mu\big(B(\xi,\rho)\cap\thickvar\LL{\beta}\butnot\thickvar\LL{\epsilon\beta}\big) \asymp_\times \mu(B(\xi,\rho)).
\end{equation}
\end{lemma}

\begin{proof}[Proof of Lemma \ref{lemmapreglobaldecay} from Lemma \ref{lemmaexistsrho1}]
Let $\rho_0 = 1$ and let $\epsilon$ be as in Lemma \ref{lemmaexistsrho1}. Fix $0 < \beta \leq 1$ and $\LL\in\Sph$. For each $\xi\in E := \Supp(\mu)\cap \thickvar\LL{\epsilon\beta}$, let $0 < \rho_\xi \leq 1$ be chosen to satisfy \eqref{existsrho1}. By the $4r$-covering lemma \cite[Theorem 8.1]{MSU}, the collection $\{B(\xi,\rho_\xi) : \xi\in E\}$ has a disjoint subcollection $\{B(\xi_i,\rho_i)\}_{i = 1}^N$ such that the collection $\{B(\xi_i,4\rho_i)\}_{i = 1}^N$ covers $E$. Then
\begin{align*}
\mu\big(\thickvar\LL{\beta}\butnot\thickvar\LL{\epsilon\beta}\big)
&\geq_\pt \sum_{i = 1}^N \mu\big(B(\xi_i,\rho_i)\cap\thickvar\LL{\beta}\butnot\thickvar\LL{\epsilon\beta}\big) \noreason\\
&\gtrsim_\times \sum_{i = 1}^N \mu(B(\xi_i,\rho_i)) \note{Lemma \ref{lemmaexistsrho1}}\\
&\asymp_\times \sum_{i = 1}^N \mu(B(\xi_i,4\rho_i)) \since{$\mu$ is Federer}\\
&\geq_\pt \mu(E) = \mu(\thickvar\LL{\epsilon\beta}).\noreason
\end{align*}
After denoting the implied constant by $C$ and letting $\lambda = 1/(1 + 1/C) < 1$, rearranging gives \eqref{preglobaldecay}.
\end{proof}

\begin{reduction}
\label{reductionP0}
In the proof of Lemma \ref{lemmaexistsrho1}, we can without loss of generality assume that there exists $\eta\in P_0$ (cf. Theorem \ref{theoremcuspfiniteness}) such that
\begin{equation}
\label{P0}
\epsilon\beta \geq \Dist^2(\xi,\eta)/C,
\end{equation}
where $C\geq 1$ is a constant possibly depending on $\rho_0$.
\end{reduction}
The basic idea is to ``pull back'' the entire picture via an isometry $g\in G$: to choose this isometry, we let $x = \xi_t$ for an appropriate value of $t$, use \eqref{GFdef} to find $H\in\scrH$ such that $x\in H$, and then let $g$ come from a top representation of $H$ (cf. Theorem \ref{theoremtoprepresentation}). After pulling back the picture, the new $\beta$ is ``on the large scale'', which translates quantitatively into the equation \eqref{P0}.
\begin{proof}
We may without loss of generality assume $\zero\in\CC_G$. Suppose that for all $C\geq 1$, Lemma \ref{lemmaexistsrho1} holds in the special case \eqref{P0}. We let $\rho_0,\epsilon,\beta,\LL,\xi,\rho$ denote the variables occurring in the version of Lemma \ref{lemmaexistsrho1} that we are trying to prove, while $\what\rho_0,\what\epsilon,\what\beta,\what\LL,\xihat,\what\rho$ denote the variables occurring in the version of Lemma \ref{lemmaexistsrho1} that we know. So fix $\rho_0 > 0$, and let $\what\rho_0 = (1\wedge\rho_0)/C_1$, where $C_1 \geq 1$ is large to be determined. Let $0 < \what\epsilon \leq 1/2$ be given, and let $\epsilon = \what\epsilon/C_1$, where $C_2 \geq C_1$ is large to be determined, possibly depending on $\what\rho_0$. Fix $0 < \beta\leq 1$, $\LL\in\Sph$, and $\xi\in \Lambda\cap \thickvar\LL{\epsilon\beta}$.

Let $t = -\log(C_2\epsilon\beta)$ and $x = \xi_t$. Then $x\in\CC_G$, since $\zero\in\CC_G$ and so by \eqref{GFdef} we have either $x\in g(B_\hyp(\0,\sigma))$ for some $g\in G$, or else $x\in H$ for some $H\in\scrH$. In the former case the reduction to the case $\what\epsilon\what\beta\geq 1/C$ follows from the comparison of the sets involved in Lemma \ref{lemmaexistsrho1} with their images under $g^{-1}$; we omit the details as they are similar to what follows. (Alternatively, when $G$ has at least one cusp, the former case can be reduced to the latter one by expanding the horoballs of $\scrH$ so that they cover $\CC_G$ rather than being disjoint.)

So suppose that $x\in H$ for some $H\in\scrH$. Let $\eta\in P_0$ and $g\in G$ be a top representation as in Theorem \ref{theoremtoprepresentation}. Let
\[
\what\beta = e^{\|g\|}\beta/C_1, \;\;
\what\LL = g^{-1}(\LL), \text{ and }
\xihat = g^{-1}(\xi).
\]
Since $\xi\in\Lambda \cap \thickvar\LL{\epsilon\beta}$, by Theorems \ref{theoremGMVT} and \ref{theoremboundeddistortion}, $\xihat\in\Lambda\cap \thickvar{\what\LL}{\what\epsilon\what\beta}$. On the other hand, since $x\in H$ and $g$ is a top representation, we have $\busemann_{g(\eta)}(\zero,x) \gtrsim_\plus \|g\|$ and thus
\[
\Dist^2(\xi,g(\eta)) \lesssim_\times e^{-t} e^{-\|g\|}.
\]
Thus we get
\[
\Dist^2(\xihat,\eta) \lesssim_\times e^{-t} e^{\|g\|} = C_2\epsilon\beta e^{\|g\|} = C_2 \what\epsilon\what\beta,
\]
so by letting $C = C_2\cdot(\text{implied constant})$, we get $\what\epsilon\what\beta \geq \Dist^2(\xihat,\eta)/C$.

Suppose first that $\what\beta\leq 1$. Then the known special case of Lemma \ref{lemmaexistsrho1} applies; let $0 < \what\rho \leq \what\rho_0$ be given as in that special case. Now since $\dist(g(\zero),\geo\zero\xi) \leq \dist(g(\zero),\geo\zero x) \asymp_\plus 0$, we have $\Dist(g^{-1}(\zero),\xihat) \gtrsim_\times 1$. Let $C_1$ be chosen so that $1/C_1 \leq (1/2)\Dist(g^{-1}(\zero,\xihat))$; then we have $\Dist(g^{-1}(\zero),\eta) \geq 1/C_1$ for all $\eta\in B(\xihat,1/C_1)$. So by Theorem \ref{theoremboundeddistortion}, after possibly increasing $C_1$ we have
\begin{equation}
\label{xiegbounds}
e^{-\|g\|} \leq |g'| \leq C_1 e^{-\|g\|} \text{ on $B(\xihat,1/C_1)$}.
\end{equation}
Since $\what\rho \leq \what\rho_0 \leq 1/C_1$, the same holds on $B(\xihat,\what\rho)$. Thus, applying $g$ to the known case of \eqref{existsrho1} and using Theorem \ref{theoremGMVT}, \eqref{pattersonsullivan}, and \eqref{xiegbounds} gives
\[
\mu\big(B(\xi,C_1 e^{-\|g\|}\what\rho)\cap\thickvar{\LL}{C_1 e^{-\|g\|}\what\beta}\butnot\thickvar{\LL}{e^{-\|g\|}\what\epsilon\what\beta}\big) \gtrsim_\times \mu(B(\xi,e^{-\|g\|}\what\rho)).
\]
Letting $\rho = C_1 e^{-\|g\|}\what\rho$ and using the Federer condition gives the desired case of \eqref{existsrho1}.

If $\what\beta > 1$, then the above argument is still valid as long as there exists $0 < \what\rho \leq \what\rho_0$ such that \eqref{existsrho1} holds. We claim that in fact, in this case \eqref{existsrho1} holds with $\what\rho = \what\rho_0$. Indeed, since $\what\beta > 1$, we have $\what\rho \leq (1 - \what\epsilon)\what\beta$ assuming $C_1\geq 2$, and thus since $\xihat\in \thickvar{\what\LL}{\what\epsilon\what\beta}$, we have
\[
B(\xihat,\what\rho) \subset \thickvar{\what\LL}{\what\beta}
\]
and thus to complete the proof, it suffices to show that
\begin{equation}
\label{LHSRHS}
\mu\big(\thickvar{\what\LL}{\what\epsilon\what\beta}\big) \leq (1/2) \mu\big(B(\xihat,\what\rho_0)\big).
\end{equation}
Now by a compactness argument, the right-hand side of \eqref{LHSRHS} is bounded below by a constant depending only on $\what\rho_0$. On the other hand, since $x\in H$ and $g$ is a top representation, we have $t\gtrsim_\plus \|g\|$ i.e. $C_2\epsilon\beta \lesssim_\times e^{-\|g\|}$ i.e. $\what\epsilon\what\beta \lesssim_\times 1/C_2$ and thus the left-hand side of \eqref{LHSRHS} is less than
\[
f(C_2) = \sup_{\what\LL\in\Sph} \mu\big(\thickvar{\what\LL}{K/C_2}\big).
\]
for some constant $K$. Now by a compactness argument and since $G$ acts irreducibly, we have $f(C_2)\to 0$ as $C_2\to\infty$. So by choosing $C_2$ sufficiently large, we can guarantee that the left-hand side of \eqref{LHSRHS} is smaller than the right-hand side.
\end{proof}

\draftnewpage

The next reduction requires some motivation. Let $\eta\in P_0$ be as in \eqref{P0}. Since $P_0$ is finite, we can treat $\eta$ as fixed. Let $\DD = \DD_\eta$ and $\PP = \PP_\eta$ be given by Theorems \ref{theoremBPP} and \ref{theoremtranslationplane}, respectively, let $k = k_\eta = \dim(\PP)$, and let $H = G_\eta$ be the stabilizer of $\eta$ in $G$. Without loss of generality we can assume that $\PP\butnot\{\eta\} \subset H(\DD)$.

We proceed to approximate $\mu$ by a smooth measure $\lambda_\eta$ on $\PP$. To motivate the choice of this measure, note that for each $h\in H$, the measure of the set $h(\DD)$ can be computed using the transformation equation \eqref{pattersonsullivan}. Thus it will be useful if the new measure $\lambda$ also satisfies \eqref{pattersonsullivan}, at least for $h\in H$. It is easy to come up with a formula for a measure on $\PP$ which satisfies \eqref{pattersonsullivan} for an even larger class of M\"obius transformations: namely, those $h\in\Isom(\H)$ such that $h(\eta) = \eta$ and $h(\PP) = \PP$. Precisely:
\begin{equation}
\label{lambdadeltadef}
\dee\lambda_\eta(\xi) = \Dist(\eta,\xi)^{2\delta - 2k} \dee\lambda_\PP(\xi).
\end{equation}
Here $\lambda_\PP$ denotes the Hausdorff $k$-dimensional measure on $\PP$ with respect to the spherical metric $\Dist = \Dist_\sph$. Note that by Theorem \ref{theorembeardon}, the singularity of $\lambda_\eta$ at $\eta$ is integrable, i.e. $\lambda_\eta$ is a finite measure.

It turns out that when approximating $\mu$ by $\lambda_\eta$, it is appropriate to replace every set of the form $\thickvar\LL\beta$ by a set of the form
\[
\w\NN(\LL,\beta) := \thickvar\LL\beta\cap B(\eta,\sqrt\beta).
\]
We can now state the next lemma in the reduction:

\begin{figure}[h!]
\begin{center}
\begin{tikzpicture}[line cap=round,line join=round,>=triangle 45,scale=2.1]
\clip(-0.2002351321003269,-1.256335988539203) rectangle (6.330320767825652,2.900189324578096);
\draw(1.944313640568794,-4.260328084250979) circle (4.896529063833537cm);
\draw(7.069865888996104,0.978921834301642) circle (2.4329337787226546cm);
\draw(2.0002058613979004,4.220495454091568) circle (3.5844786492188563cm);
\draw(4.124979893500483,0.7030984897295431) circle (0.5248716852402419cm);
\draw(3.3721858839985472,0.6623294912630658) circle (0.22904312170249935cm);
\draw(3.015078543426337,0.6491004616872263) circle (0.12831394317950878cm);
\draw(2.8048697954632837,0.6433315384932827) circle (0.08207161680921136cm);
\draw(2.6658151126991405,0.640352281620797) circle (0.057015400541110095cm);
\draw(2.5669310830162484,0.6386838974707642) circle (0.041890466124819066cm);
\draw[thick] (-2.13726402256241,1.60638591031447)-- (11.306569205315798,-0.9414977560616642);
\draw[dashed] [domain=-0.2002351321003269:6.330320767825652] plot(\x,{(--17.238976472802307-2.5478836663761344*\x)/13.443833227878208});
\draw[dashed] [domain=-0.2002351321003269:6.330320767825652] plot(\x,{(--15.131284443865017-2.5478836663761344*\x)/13.443833227878208});
\begin{scriptsize}
\draw [fill=black] (1.9765831411979147,0.6360946458955603) circle (0.5pt);
\draw[color=black] (1.98,0.50) node {$\eta$};
\end{scriptsize}
\end{tikzpicture}
\caption{A schematic drawing of a hyperplane-neighborhood near a rank one cusp $\eta$. The balls represent sets $h(\DD)$, $h\in G_\eta$. The measure of the hyperplane-neighborhood is estimated by considering its intersection with each of the sets $h(\DD)$.}
\label{figureHello}
\end{center}
\end{figure}

\begin{lemma}
\label{lemmaexistsrho2}
For all $\rho_0 > 0$, there exists $\epsilon > 0$ such that for all $0 < \beta \leq \rho_0^2/4$, $\LL\in\Sph$, and $\xi\in \PP\cap \w\NN(\LL,\epsilon\beta)$, there exists $\beta/2 \leq \rho \leq \rho_0$ such that
\begin{equation}
\label{existsrho2}
\lambda_\eta\big(B(\xi,\rho)\cap \w\NN(\LL,\beta)\butnot \w\NN(\LL,\epsilon\beta)\big) \asymp_\times \lambda_\eta\big(B(\xi,\rho)\big).
\end{equation}
\end{lemma}
In the proof of Lemma \ref{lemmaexistsrho1} from Lemma \ref{lemmaexistsrho2}, we will frequently use the asymptotic
\begin{equation}
\label{hDprime}
|h'| \asymp_\times e^{-\|h\|} \asymp_\times \Dist^2(\eta,h(\DD)) \text{ on $\DD$ ($h\in H$)},
\end{equation}
which can be proven by conjugating $\eta$ to $\infty$ and then comparing the Euclidean and spherical metrics (noting that in the Euclidean metric, elements of $H$ act as isometries). Note that by applying Theorem \ref{theoremGMVT} to \eqref{hDprime}, we get
\begin{equation}
\label{hDdiam}
\Diam(h(\DD)) \asymp_\times \Dist^2(\eta,h(\DD)).
\end{equation}
\begin{proof}[Proof of Lemma \ref{lemmaexistsrho1} using Lemma \ref{lemmaexistsrho2}]
We let $\rho_0,\epsilon,\beta,\LL,\xi,\rho$ denote the variables appearing in the desired Lemma \ref{lemmaexistsrho1}, and we let $\what\rho_0,\what\epsilon,\what\beta,\what\LL,\xihat,\what\rho$ denote the variables appearing in the known Lemma \ref{lemmaexistsrho2}. Fix $\rho_0 > 0$, choose $\what\rho_0 \leq \rho_0/3$ to be determined, let $\what\epsilon > 0$ be given, and let $\epsilon = \what\epsilon/C^2 > 0$, where $C \geq 1$ is large to be determined. Fix $0 < \beta \leq 1$, $\LL\in\Sph$, and $\xi\in \Lambda\cap \NN(\LL,\epsilon\beta)$, let $\what\beta = \beta/C$ and $\what\LL = \LL$, and note that
\[
\epsilon\beta < \what\epsilon\what\beta < \what\beta < \beta.
\]
By \eqref{hDdiam}, we may choose $\xihat\in\PP$ so that $\Dist(\xi,\xihat) \lesssim_\times \Dist^2(\eta,\xi)$. Combining with \eqref{P0} gives
\begin{align*}
\Dist(\eta,\xihat) &\lesssim_\times \sqrt{\epsilon\beta},&
\Dist(\xihat,\what\LL) &\lesssim_\times \epsilon\beta.
\end{align*}
So by choosing $C$ sufficiently large, we can guarantee that $\xihat\in \w\NN(\what\LL,\what\epsilon\what\beta)$. Then we can let $\what\beta/2 \leq \what\rho \leq \what\rho_0$ be given as in Lemma \ref{lemmaexistsrho2}. Let $\rho = 3\what\rho \leq 3\what\rho_0 = \rho_0$. To complete the proof we need to show:
\begin{align} \label{geometric1}
\mu\big(B(\xi,\rho)\cap \NN(\LL,\beta)\butnot \NN(\LL,\epsilon\beta)\big) &\gtrsim_\times \lambda_\eta\big(B(\xihat,\what\rho)\cap \w\NN(\what\LL,\what\beta)\butnot \w\NN(\what\LL,\what\epsilon\what\beta)\big)\\ \label{geometric2}
\mu\big(B(\xi,\rho)\big) &\lesssim_\times \lambda_\eta\big(B(\xihat,5\what\rho)\big).
\end{align}
(\eqref{geometric2} suffices since $\lambda_\eta$ is Federer.) Fix $h\in H$, and we will demonstrate \eqref{geometric1}-\eqref{geometric2} via their intersections with $h(\DD)$, see Figure \ref{figureHello}, i.e. we will show that
\begin{align} \label{geometric3}
\mu\big(h(\DD)\cap B(\xi,\rho)\cap \NN(\LL,\beta)\butnot \NN(\LL,\epsilon\beta)\big) &\gtrsim_\times \lambda_\eta\big(h(\DD)\cap B(\xihat,\what\rho)\cap \w\NN(\what\LL,\what\beta)\butnot \w\NN(\what\LL,\what\epsilon\what\beta)\big)\\ \label{geometric4}
\mu\big(h(\DD)\cap B(\xi,\rho)\big) &\lesssim_\times \lambda_\eta\big(h(\DD)\cap B(\xihat,5\what\rho)\big).
\end{align}
The following consequence of \eqref{pattersonsullivan} and \eqref{hDprime} will be useful in proving both \eqref{geometric3} and \eqref{geometric4}:
\begin{equation}
\label{mulambdaasymp}
\mu(h(\DD)) \asymp_\times e^{-\delta\|h\|} \asymp_\times \Dist^{2\delta}(\eta,h(\DD)) \asymp_\times \lambda_\eta(h(\DD)).
\end{equation}
Also, by \eqref{P0}, if we choose $C$ sufficiently large then
\begin{equation}
\label{xiwhatxi}
\Dist(\xi,\xihat) \leq \what\rho.
\end{equation}

\begin{subproof}[Proof of \eqref{geometric4}]
To avoid trivialities, suppose that $h(\DD)\cap B(\xi,\rho) \neq \emptyset$. Then by \eqref{P0},
\[
\Dist(\eta,h(\DD)) \leq \Dist(\eta,\xi) + \rho \lesssim_\times \sqrt{\epsilon\beta} + \rho.
\]
Applying \eqref{hDdiam} gives
\begin{align*}
\Diam(h(\DD)) &\lesssim_\times \epsilon\beta + \rho^2 = \what\epsilon\what\beta/C + 4\what\rho^2\\
&\leq_\pt (2\what\epsilon/C + 4\what\rho_0) \what\rho \since{$\what\beta/2 \leq \what\rho \leq \what\rho_0$}
\end{align*}
and thus by choosing $\what\rho_0$ small enough and $C$ large enough, we get $\Diam(h(\DD)) \leq \what\rho$. Combining with \eqref{xiwhatxi} gives $h(\DD) \subset B(\xihat,5\what\rho)$, and \eqref{mulambdaasymp} completes the proof.
\end{subproof}
\begin{subproof}[Proof of \eqref{geometric3}]
To avoid trivialities, suppose that
\begin{equation}
\label{avoidtrivialities}
h(\DD)\cap B(\xihat,\what\rho)\cap \w\NN(\what\LL,\what\beta)\butnot \w\NN(\what\LL,\what\epsilon\what\beta)\neq\emptyset.
\end{equation}
By \eqref{xiwhatxi} we have $B(\xihat,\what\rho) \subset B(\xi,\rho)$ and thus $h(\DD)\cap B(\xi,\rho) \neq \emptyset$, so the above argument shows that $\Diam(h(\DD)) \leq \what\rho$ and combining with \eqref{xiwhatxi} gives $h(\DD) \subset B(\xi,\rho)$. On the other hand, since $h(\DD)\cap B(\eta,\what\beta^{1/2}) \neq \emptyset$, we have $\Dist(\eta,h(\DD)) \leq \what\beta^{1/2}$ and thus by \eqref{hDdiam}, $\Diam(h(\DD)) \lesssim_\times \what\beta$. If $C$ is sufficiently large, then this implies $\Diam(h(\DD)) \leq \beta/2$ and thus $h(\DD) \subset \NN(\LL,\beta)$. So by \eqref{mulambdaasymp}, to complete the proof it is enough to show that
\[
\mu\big(h(\DD)\butnot \NN(\LL,\epsilon\beta)\big) \asymp_\times \mu\big(h(\DD)\big).
\]
Again to avoid trivialities, let us assume that
\begin{equation}
\label{avoidtrivialities2}
h(\DD) \cap \NN(\what\LL,\epsilon\beta) \neq \emptyset.
\end{equation}
From \eqref{avoidtrivialities}, we have $h(\DD) \butnot \w\NN(\what\LL,\what\epsilon\what\beta) \neq \emptyset$, so either
\begin{equation}
\label{geometricoption1}
h(\DD) \butnot \NN(\what\LL,\what\epsilon\what\beta) \neq \emptyset
\end{equation}
or
\begin{equation}
\label{geometricoption2}
h(\DD) \butnot B\big(\eta,(\what\epsilon\what\beta)^{1/2}\big) \neq \emptyset.
\end{equation}
Combining \eqref{avoidtrivialities2}, \eqref{geometricoption1}, and \eqref{hDdiam} gives $e^{-\|h\|} \asymp_\times \Diam(h(\DD)) \geq \what\epsilon\what\beta - \epsilon\beta$, while combining \eqref{geometricoption2} and \eqref{hDdiam} gives $e^{-\|h\|} \asymp_\times \Dist^2(\eta,h(\DD)) \gtrsim_\times \what\epsilon\what\beta$. Either way we get $e^{-\|h\|} \gtrsim_\times \what\epsilon\what\beta$, and thus it suffices to show
\[
\mu\big(h(\DD)\butnot \NN(\LL,K e^{-\|h\|}/C)\big) \asymp_\times \mu\big(h(\DD)\big)
\]
where $K > 1$ is a constant. By \eqref{hDprime} and \eqref{pattersonsullivan}, it is enough to show that
\begin{equation}
\label{ETSgeometricreduction}
\mu\big(\DD\butnot \NN(h^{-1}(\LL),K/C)\big) \asymp_\times \mu(\DD).
\end{equation}
We proceed by contradiction; if no $C$ exists satisfying \eqref{ETSgeometricreduction} (for all $\LL\in\Sph$ and $h\in H$), then a compactness argument proves the existence of $\LL_0\in\Sph$ such that $\mu(\DD\butnot\LL_0) = 0$. A zooming argument shows that we may take $\LL_0$ so that $\mu(\Lambda\butnot\LL_0) = 0$, i.e. $\Lambda\subset\LL_0$. But this contradicts the assumption that $G$ is irreducible.
\end{subproof}
This completes the proof of Lemma \ref{lemmaexistsrho1} modulo Lemma \ref{lemmaexistsrho2}.
\end{proof}

\begin{reduction}
\label{reductionsphtoeuc}
In the proof of Lemma \ref{lemmaexistsrho2}, we can without loss of generality assume that $\PP$ is a plane (rather than a sphere) and that $\eta = \0$, and we can work in the Euclidean metric rather than the spherical metric.
\end{reduction}
\begin{proof}
The first reduction follows by applying a fixed conjugation in which we move $\eta$ to $\0$ and some other point of $\PP$ to $\infty$. The second reduction follows from choosing $\rho_0$ small enough so that $\Dist_\euc \asymp_\times \Dist_\sph$ on $B_\sph(\eta,\rho_0)$, and modifying the constants $\epsilon,\beta,\rho$ appropriately.
\end{proof}

We are now ready to finish the proof of Theorem \ref{theoremGFQD} by proving the Euclidean version of Lemma \ref{lemmaexistsrho2}:

\begin{figure}[h!]
\begin{center}
\begin{tikzpicture}[line cap=round,line join=round,>=triangle 45,x=1.0cm,y=1.0cm]
\clip(-5.681385247265749,-0.7580748479629699) rectangle (5.807678099921845,5.0295209924730315);
\draw[thick](0.0,-0.0) circle (3.0cm);
\draw [dash pattern=on 2pt off 2pt] (0.0,-0.0) circle (4.41cm);
\draw [shift={(0.0,-0.0)},dash pattern=on 1pt off 1pt]  plot[domain=0.21391741110522272:2.9422391174694305,variable=\t]({1.0*2.6960221843722216*cos(\t r)+-0.0*2.6960221843722216*sin(\t r)},{0.0*2.6960221843722216*cos(\t r)+1.0*2.6960221843722216*sin(\t r)});
\draw [shift={(0.0,-0.0)},dash pattern=on 1pt off 1pt]  plot[domain=0.19983276794005583:2.9816716353585653,variable=\t]({1.0*3.299919748638474*cos(\t r)+-0.0*3.299919748638474*sin(\t r)},{0.0*3.299919748638474*cos(\t r)+1.0*3.299919748638474*sin(\t r)});
\begin{scriptsize}
\draw [fill=black] (0.0,-0.0) circle (0.75pt);
\draw[color=black] (0.002163926375931373,-0.19882401394331134) node {$y$};

\draw[color=black] (-4.03,3.03) node {$\thickvar\LL{\beta}$};
\draw[color=black] (-1.3,3.42) node {$\thickvar\LL{\epsilon\beta}$};
\draw[color=black] (-3.24,-0.45) node {$\LL$};

\draw [fill=black] (0.0,2.0) circle (0.75pt);
\draw[color=black] (-0.02384774032265756,1.8170801551973184) node {$y_1$};

\draw [fill=black] (2.01,2.0) circle (0.75pt);
\draw[color=black] (2.15,1.85) node {$\xi$};
\draw(2.01,2.0) circle (1.8cm);

\draw [fill=black] (3.36,2.0) circle (0.75pt);
\draw[color=black] (3.4,1.8) node {$\zeta$};
\draw(3.36,2.0) circle (0.45cm);

\draw (-6,2) -- (6,2);

\draw[arrows=<-,line width=.1pt] (3.55,1.65) -- (4.6,0.18);

\draw[color=black] (3.4,3.7) node {$B(\xi,\rho)$};
\draw[color=black] (5.2,0.17) node {\tiny $B(\zeta,\rho/4)$};
\end{scriptsize}
\end{tikzpicture}
\caption{Various entities appearing in the proof of Lemma \ref{lemmaexistsrho2}, illustrating the inclusion \eqref{geometricv2} in the case $\rho = \Dist(\xi,\PP\butnot \thickvar\LL\beta)$.}
\label{figurelemmaexistsrho2}
\end{center}
\end{figure}

\begin{proof}[Proof of Lemma \ref{lemmaexistsrho2}]
Let $\rho = \Dist(\xi,\PP\butnot \w\NN(\LL,\beta))$. Then
\[
\rho \geq \Dist(\w\NN(\LL,\epsilon\beta),\PP\butnot\w\NN(\LL,\beta)) \geq (\beta - \epsilon\beta) \wedge(\sqrt\beta - \sqrt{\epsilon\beta}) \geq \beta/2
\]
and on the other hand, $\rho \leq \Diam(\w\NN(\LL,\beta)) \leq 2\sqrt\beta \leq \rho_0$. By construction, $B(\xi,\rho)\cap\PP \subset \w\NN(\LL,\beta)$. Since the measure $\lambda_\eta$ is Federer, to prove \eqref{existsrho2} it suffices to show that there exists a ball $B(\zeta,\rho/4)$ such that $\zeta\in\PP$ and
\begin{equation}
\label{geometricv2}
B(\zeta,\rho/4)\cap \PP \subset B(\xi,\rho)\butnot \w\NN(\LL,\epsilon\beta), 
\end{equation}
see Figure \ref{figurelemmaexistsrho2}. We consider two cases:
\begin{itemize}
\item Suppose $\rho = \Dist(\xi,\PP\butnot B(\0,\sqrt\beta)) = \sqrt\beta - \|\xi\|$. Let $\vv\in\PP$ be a unit vector in the direction of $\xi$ (in an arbitrary direction if $\xi = \0$) and let $\zeta = \xi + 3\rho\vv/4$. Then for all $\zz\in B(\zeta,\rho/4)$,
\[
\|\zz\| \geq \|\zeta\| - \rho/4 \geq \|\xi\| + \rho/2 \geq \sqrt\beta/2 \geq \sqrt{\epsilon\beta}
\]
assuming $\epsilon\leq 1/4$. Thus $B(\zeta,\rho/4)\cap B(\0,\sqrt{\epsilon\beta}) = \emptyset$, demonstrating \eqref{geometricv2}.
\item Suppose $\rho = \Dist(\xi,\PP\butnot \NN(\LL,\beta))$. Then $\PP\nsubset \LL$, since otherwise $\rho = \infty$ which contradicts the definition of $\rho$. Suppose that $\LL$ is a sphere (the case where $\LL$ is a hyperplane is easier and will be omitted), and write $\LL = \{\xx : \|\yy - \xx\| = k\}$ for some $\yy\in\R^d$ and $k > 0$. Write $\yy = \yy_1 + \yy_2$ with $\yy_1\in \PP$ and $\yy_2\in \PP^\perp$. Let $\vv\in\PP$ be a unit vector in the direction of $\xi - \yy_1$ (in an arbitrary direction if $\xi = \yy_1$), and let $\zeta = \xi + 3\rho\vv/4$. Then for all $\zz\in B(\zeta,\rho/4)\cap \PP$,
\[
\Dist(\zz,\LL) = \|\zz - \yy\| - k \geq \sqrt{(\|\xi - \yy_1\| + \rho/2)^2 + \|\yy_2\|^2} - k.
\]
We aim to show that the right hand side exceeds $\epsilon\beta$. Indeed, since $\xi\in \w\NN(\LL,\epsilon\beta)$ we have
\[
\epsilon\beta \geq \Dist(\xi,\LL) \geq k - \|\xi - \yy\| = k - \sqrt{\|\xi - \yy_1\|^2 + \|\yy_2\|}
\]
and since $\xi + \rho\vv\in B(\xi,\rho) \subset \del\w\NN(\LL,\beta)$, we have
\[
\beta = \Dist(\xi + \rho\vv,\LL) = \|\xi + \rho\vv\| - k = \sqrt{(\|\xi - \yy_1\| + \rho)^2 + \|\yy_2\|^2} - k,
\]
i.e.
\begin{align*}
\sqrt{\|\xi - \yy_1\|^2 + \|\yy_2\|^2} - k &\geq -\epsilon\beta\\
\sqrt{(\|\xi - \yy_1\| + \rho)^2 + \|\yy_2\|^2} - k &\geq \beta
\end{align*}
which implies
\[
\sqrt{(\|\xi - \yy_1\| + \rho/2)^2 + \|\yy_2\|^2} - k \geq (1/4)\beta + (3/4)(-\epsilon\beta).
\]
(Notice that the general inequality \[\sqrt{(a + \delta x)^2 + b^2} \geq \delta^2 \sqrt{(a + x)^2 + b^2} + (1 - \delta^2) \sqrt{a^2 + b^2}\] ($a,b,x\geq 0$, $0\leq \delta\leq 1$) can be verified by first checking the inequality \[(a + \delta x)^2 + b^2 \geq \delta^2 [(a + x)^2 + b^2] + (1 - \delta^2) [a^2 + b^2]\] and then using the downward convexity of the square root function.) 
Assuming $\epsilon\leq 1/7$, this gives $\Dist(\zz,\LL) \geq \epsilon\beta$ and thus \eqref{geometricv2} holds.
\qedhere\end{itemize}
\end{proof}

\draftnewpage
\subsection{Proof of friendliness}
Like the proof of Lemma \ref{lemmaglobaldecay}, the proof of Theorem \ref{theoremGFfriendly} proceeds via a series of reductions, in which the final question is about the smooth measure $\lambda_\0$ defined above, and does not depend on the initial measure $\mu$. Since many of the arguments are similar to those in the proof of Lemma \ref{lemmaglobaldecay}, we will not provide the full details of the reductions.

We will actually prove a result which is slightly stronger than friendliness, namely ``friendliness to spheres'':

\begin{theorem}[Friendliness to spheres]
\label{theoremGFfriendlyspheres}
There exists $\alpha > 0$ such that for all $\xi\in\Lambda$, $0 < \rho \leq 1$, $\beta > 0$, and $\LL\in\Sph$, if $B = B(\xi,\rho)$ then
\[
\mu\big(\thickvar\LL{\beta \|d_\LL\|_{\mu,B}}\cap B\big) \lesssim_\times \beta^\alpha\mu(B),
\]
and if $k_{\min} = d$ then
\[
\mu\big(\thickvar\LL{\beta \rho}\cap B\big) \lesssim_\times \beta^\alpha\mu(B).
\]
\end{theorem}
Note that Theorem \ref{theoremGFfriendlyspheres} implies the hard direction of Theorem \ref{theoremGFfriendly}. The easy direction is proven as follows: Suppose that $k_\eta < d$ for some $\eta$, and we will show that $\mu$ is not absolutely decaying. By Theorems \ref{theoremBPP} and \ref{theoremtranslationplane}, there exist two generalized spheres $\LL_1,\LL_2\in \Sph$ which are tangent at $\eta$ such that $\Lambda$ is contained in the region between $\LL_1$ and $\LL_2$. Let $\LL$ be the hyperplane tangent to both $\LL_1$ and $\LL_2$ at $\eta$. Then for all $\rho > 0$,
\[
\Lambda\cap B(\eta,\rho) \subset \thickvar\LL{C\rho^2},
\]
where $C > 0$ is a large constant. This implies that $\mu$ is not absolutely decaying (even stronger, $\Lambda$ is not hyperplane diffuse in the sense of \cite[Definition 4.2]{BFKRW}).

We now proceed with the proof of Theorem \ref{theoremGFfriendlyspheres}.

\begin{reduction}
\label{reductionP0friendly}
In the proof of Theorem \ref{theoremGFfriendlyspheres}, we can without loss of generality assume that
\begin{equation}
\label{P0friendly}
\rho \gtrsim_\times \Dist^2(\xi,\eta)
\end{equation}
for some $\eta\in P_0$, where $P_0$ is as in Theorem \ref{theoremcuspfiniteness}.
\end{reduction}
\begin{proof}
Similar to the proof of Reduction \ref{reductionP0}.
\end{proof}

As in the proof of Lemma \ref{lemmaglobaldecay}, fix $\eta\in P_0$, and let $\DD = \DD_\eta$, $\PP = \PP_\eta$, $k = k_\eta = \dim(\PP)$, $H = G_\eta$, and $\lambda_\eta$ be as before. We will prove the reduction of Theorem \ref{theoremGFfriendlyspheres} to the following lemma in a manner analogous to the reduction of Lemma \ref{lemmaexistsrho1} to Lemma \ref{lemmaexistsrho2}:

\begin{lemma}
\label{lemmaGFfriendlyleb}
There exists $\alpha_0 > 0$ such that for all $0 < \alpha \leq \alpha_0$, $\xi\in\PP$, $0 < \rho \leq 1$, $\beta > 0$, and $\LL\in\Sph$, if $B = B(\xi,\rho)$ and
\begin{align*}
\sigma_1 &= \max_B \Dist^2(\eta,\cdot),&
\sigma_2 &= \|d_\LL\|_{B\cap\PP},
\end{align*}
then either
\begin{equation}
\label{GFfriendly1}
\int_{B\cap \PP} \left(\frac{\beta(\sigma_1\vee\sigma_2)}{\Dist^2(\eta,\yy)}\right)^\alpha \;\dee\lambda_\eta(\yy) \lesssim_\times \beta^{\alpha/2} \lambda_\eta\big(B\big)
\end{equation}
or
\begin{equation}
\label{GFfriendly2}
\lambda_\eta\big(\thickvar\LL{K\sigma_1\vee\beta\sigma_2}\cap B\big) \lesssim_\times \beta^{\alpha/2} \lambda_\eta\big(B\big),
\end{equation}
where $K > 0$ is the implied constant of \eqref{hDdiam}.
\end{lemma}
\begin{proof}[Proof of Theorem \ref{theoremGFfriendlyspheres} using Lemma \ref{lemmaGFfriendlyleb}]
We let $\alpha,\xi,\rho,\beta,\LL$ denote the variables appearing in the desired Theorem \ref{theoremGFfriendlyspheres}, and we let $\what\alpha_0,\what\alpha,\xihat,\what\rho,\what\beta,\what\LL$ denote the variables appearing in the known Lemma \ref{lemmaGFfriendlyleb}. Let $\what\alpha_0 > 0$ be given, let $\w\alpha > 0$ be as in Lemma \ref{lemmaglobaldecay}, and let $\alpha = \what\alpha = \what\alpha_0 \wedge\w\alpha > 0$. Fix $\xi\in\Lambda$, $\Dist^2(\xi,\eta) \lesssim_\times \rho \leq 1$, $\beta > 0$, and $\LL\in\Sph$, and let $B = B(\xi,\rho)$. By \eqref{hDdiam}, we may choose $\xihat\in\PP$ so that $\Dist(\xi,\xihat) \lesssim_\times \Dist^2(\eta,\xi) \lesssim_\times \rho$. Fix $C \geq 1$ large to be determined, and let $\what\rho = C\rho > 0$, so that $B(\xi,\rho) \subset \what B := B(\xihat,\what\rho)$ assuming $C$ is large enough. Let $\what\beta = C\beta$ and $\what\LL = \LL$. Finally, let
\begin{align*}
\sigma &= \begin{cases}
\|d_\LL\|_{\mu,B} & k_{\min} < d\\
\rho & k_{\min} = d
\end{cases},&
\what\sigma_1 &= \max_{\what B} \Dist^2(\eta,\cdot),&
\what\sigma_2 &= \|d_\LL\|_{\what B\cap\PP}.
\end{align*}
Note that by \eqref{hDdiam}, we have
\begin{equation}
\label{sigma12}
\sigma \lesssim_\times \what\sigma_1 \vee \what\sigma_2.
\end{equation}
(When $k_{\min} = d$, \eqref{sigma12} follows from the asymptotic $\what\sigma_2 \asymp_\times \what\rho \geq \rho$.) Since $\mu$ is Federer, to complete the proof it suffices to show that
\begin{align} \label{friendlyred1a}
\mu\big(\thickvar\LL{\beta \sigma}\cap B\big) &\lesssim_\times \int_{\what B\cap \PP} \left(\frac{\what\beta(\what\sigma_1\vee\what\sigma_2)}{\Dist^2(\eta,\yy)}\right)^{2\alpha} \;\dee\lambda_\eta(\yy)\\ \label{friendlyred1b}
\mu\big(\thickvar\LL{\beta \sigma}\cap B\big) &\lesssim_\times \lambda_\eta\big(\thickvar{\what\LL}{K\what\sigma_1\vee\what\beta\what\sigma_2}\cap \what B\big)\\ \label{friendlyred2}
\mu\big(B(\xi,2\what\rho)\big) &\gtrsim_\times \lambda_\eta\big(\what B\big)
\end{align}
Fix $h\in H$, and we will prove \eqref{friendlyred1a}-\eqref{friendlyred2} via their intersections with $h(\DD)$. The proof of \eqref{friendlyred2} is similar to the proof of \eqref{geometric2}. Suppose $h(\DD) \cap \thickvar\LL{\beta\sigma}\cap B\neq\emptyset$, so that by \eqref{hDdiam} and \eqref{P0friendly}, we have $h(\DD)\subset \thickvar{\what\LL}{K\what\sigma_1\vee\what\beta\what\sigma_2}\cap \what B$. Then \eqref{friendlyred1b} follows from \eqref{mulambdaasymp}, and \eqref{friendlyred1a} is reduced to
\begin{equation}
\label{friendlyred3}
\mu\big(\thickvar\LL{\beta \sigma}\cap h(\DD)\big) \lesssim_\times (e^{\|h\|}\what\beta(\what\sigma_1\vee\what\sigma_2))^{2\alpha} \lambda_\eta\big(h(\DD)\big).
\end{equation}
Applying \eqref{hDprime}, \eqref{GMVT}, and \eqref{pattersonsullivan} reduces us to proving
\begin{equation}
\label{friendlyred4}
\mu\big(\thickvar{h^{-1}(\LL)}{K_2 e^{\|h\|} \beta \sigma}\cap \DD\big) \lesssim_\times (e^{\|h\|}\what\beta(\what\sigma_1\vee\what\sigma_2))^{2\alpha} \lambda_\eta\big(\DD\big),
\end{equation}
where $K_2 > 1$ is the implied constant of \eqref{hDprime}. But this follows from \eqref{globaldecay} and \eqref{sigma12}.
\end{proof}

\begin{reduction}
In the proof of Lemma \ref{lemmaGFfriendlyleb}, we can without loss of generality assume that $\PP$ is a plane (rather than a sphere) and that $\eta = \0$, and we can work in the Euclidean metric rather than the spherical metric.
\end{reduction}
\begin{proof}
Similar to the proof of Reduction \ref{reductionsphtoeuc}.
\end{proof}

\begin{proof}[Proof of Lemma \ref{lemmaGFfriendlyleb}]
It suffices to prove
\begin{align} \label{GFfriendly1v2}
\int_{B\cap\PP} \left(\frac{\sigma_1}{\Dist^2(\0,\yy)}\right)^\alpha \;\dee\lambda_\0(\yy) &\lesssim_\times \lambda_\0(B)\\ \label{GFfriendly2v2}
\lambda_\0\big(\thickvar\LL{2K\beta^{1/2}\sigma_2}\cap B\big) &\lesssim_\times (2K\beta^{1/2})^\alpha \lambda_\0(B)
\end{align}
since the former implies \eqref{GFfriendly1} if $\sigma_1 \geq \beta^{1/2}\sigma_2$, and the latter implies \eqref{GFfriendly2} if $\sigma_1 \leq \beta^{1/2}\sigma_2$. Using the asymptotic
\[
\lambda_\0 \given B(\0,1) \asymp_\times \sum_{n\in\N} 2^{n(2k - 2\delta)} \lambda_\PP^{(\euc)}\given B(\0,2^{-n})\butnot B(\0,2^{-(n + 1)}),
\]
one can show that it suffices to consider the case where $\Dist(\0,B) \geq \rho$. Here $\lambda_\PP^{(\euc)}$ is the Euclidean Lebesgue measure on $\PP$. Now if $\Dist(\0,B) \geq \rho$, then $\Dist(\0,\yy) \asymp_\times \Dist(\0,B) \asymp_\times \sigma_1^{1/2}$ for all $\yy\in B$, which implies \eqref{GFfriendly1v2}.
\end{proof}

\begin{lemma}
\label{lemmalebesguefriendly}
If $\lambda$ is Lebesgue measure on a $k$-dimensional subspace $\PP\subset\R^d$, $\LL\in\Sph$, and $B(\xi,\rho)\subset\PP$, then
\begin{equation}
\label{lebesguefriendly}
\lambda\big(\thickvar\LL{\beta\sigma}\cap B(\xi,\rho)\big) \lesssim_\times \beta^\alpha \lambda\big(B(\xi,\rho)\big),
\end{equation}
where $\sigma = \|d_\LL\|_{B(\xi,\rho)\cap\PP}$, and $\alpha > 0$ is a uniform constant.
\end{lemma}
\begin{proof}
Without loss of generality suppose that $\|\xi\|,\rho \leq 1$. Then we can use the spherical metric instead of the Euclidean metric. Using the spherical metric, \eqref{lebesguefriendly} is just the assertion that the image of Lebesgue measure under stereographic projection onto the sphere is friendly. But this follows from \cite[Theorem 2.1]{KLW}.
\end{proof}

\draftnewpage
\section{Gibbs states of CIFSes}
\label{sectiongibbs}

We will prove Theorem \ref{theoremgibbs} by first proving a general theorem about measures which come from ``coding maps'' (Theorem \ref{theoremIFSgeneral}). This general theorem will be useful in later papers in this series, where we will use it to deduce that certain random measures are quasi-decaying.

\subsection{A general theorem}
Before stating the general theorem, we state a lemma which essentially says that when a measure is defined as the image of another measure, then the quasi-decaying condition can be checked on the level of the original measure rather than on the level of the image measure:

\begin{lemma}
\label{lemmaQDETS}
Let $(X,\mu)$ be a measure space, and let $\pi:X\to\R^d$ be a measurable map. Suppose that there exists a sequence of sets $(\borel_n)_1^\infty$ in $X$ such that $\mu(X\butnot\bigcup_n \borel_n) = 0$ and for all $n\in\N$, for $\mu$-a.e. $x\in \borel_n$, for all $\gamma > 0$, there exist $C_1,\alpha > 0$ such that for all $0 < \rho \leq 1$, $0 < \beta \leq \rho^\gamma$, and $\LL\in\Hyp$, if $B = B(\pi(x),\rho)$ then
\begin{equation}
\label{QDETS}
\mu\big(\pi^{-1}(\NN(\LL,\beta\rho)\cap B)\cap \borel_n\big) \leq C_1 \beta^\alpha \mu\big(\pi^{-1}(B)\big).
\end{equation}
Then $\wbar\mu := \pi_*[\mu]$ is quasi-decaying.
\end{lemma}
\begin{proof}
Without loss of generality assume that the sequence $(E_n)_1^\infty$ is increasing. For each $n$, define the measure $\nu_n$ on $\R^d$ via the formula
\[
\nu_n(S) = \mu\big(\pi^{-1}(S)\cap \borel_n\big),
\]
i.e. $\nu_n = \pi_*[\mu\given \borel_n]$. Then $\nu_n \nearrow \wbar\mu$. Let $f_n$ denote the Radon--Nikodym derivative $\frac{\dee\nu_n}{\dee\wbar\mu}$, so that $f_n\nearrow 1$ $\wbar\mu$-a.e. We may choose $f_n$ so that $f_n = 0$ on $\R^d\butnot \pi(\borel_n)$. Finally, let $F_n = \{\yy\in \R^d : f_n(\yy) \geq 1/2\}$, so that $\wbar\mu(\R^d \butnot \bigcup_n F_n) = 0$. For each $n$, we can see that $\wbar\mu$ is quasi-decaying relative to $F_n$ as follows: given $\yy \in F_n$ and $\gamma > 0$, choose $x\in \borel_n$ so that $\pi(x) = \yy$, and let $C_1,\alpha > 0$ be as in the hypothesis of the lemma. Then for all $0 < \rho \leq 1$, $0 < \beta \leq \rho^\gamma$, and $\LL\in\Hyp$, if $B = B(\yy,\rho)$ then
\begin{align*}
\wbar\mu\big(\NN(\LL,\beta\rho)\cap B\cap F_n\big)
&\leq 2 \nu_n\big(\NN(\LL,\beta\rho)\cap B\big)\\
&= 2 \mu\big(\pi^{-1}(\NN(\LL,\beta\rho)\cap B)\cap \borel_n\big)\\
&\leq 2 C_1 \beta^\alpha \mu\big(\pi^{-1}(B)\big)
= 2 C_1 \beta^\alpha \wbar\mu(B).
\qedhere\end{align*}
\end{proof}
\begin{notation}
In the sequel we will not distinguish between the sets $S$ and $\pi^{-1}(S)$, so for example formula \eqref{QDETS} would be written
\[
\mu\big(\NN(\LL,\beta\rho)\cap B\cap \borel_n\big) \leq C_1 \beta^\alpha \mu(B).
\]
\end{notation}

We now state our general theorem about measures coming from ``geometrically nice'' coding maps.

\begin{theorem}
\label{theoremIFSgeneral}
Let $\alphabet$ be a measurable space, let $\pi:\alphabet^\N\to\R^d$ be a measurable map, and let $\mu$ be a probability measure on $\alphabet^\N$. For each $\omega\in \alphabet^*$, we let $\mu_\omega$ denote the conditional measure of $\mu$ on the cylinder $[\omega] = \{\tau\in \alphabet^\N : \tau\given |\omega| = \omega\}$ (normalized to be a probability measure), and we fix a real number $D_\omega \geq \Diam(\Supp(\mu_\omega))$. Assume that $D_\tau \leq D_\omega$ whenever $\tau$ extends $\omega$. Fix $\kappa > 0$, $r\in\N$, and a set
\begin{equation}
\label{Gkrdef}
G \subset \{\omega\in \alphabet^* : \mu_\omega(\{\tau\in \alphabet^\N : \Supp(\mu_{\tau\given |\omega| + r})\cap \NN(\LL,\kappa D_\omega) = \emptyset\}) \geq \kappa \all \LL\in\Hyp\}.
\end{equation}
Assume that for $\mu$-a.e. $\omega\in \alphabet^\N$, the limits
\begin{align} \label{klimit}
&\lim_{n\to\infty} \frac1n \#\{i = 1,\ldots,n : \omega\given i \in G\}\\ \label{dlimit}
&\lim_{n\to\infty} \frac1n \log(1/D_{\omega\given n})
\end{align}
exist and are positive. Then $\pi_*[\mu]$ is quasi-decaying.
\end{theorem}
We will prove the theorem for the case $r = 1$; the general case follows by replacing $\alphabet$ by $\alphabet^r$.

The basic idea of the proof is as follows: Suppose that we are given a $\mu$-random word $\tau\in \alphabet^\N$, of which we know an initial segment $\omega = \tau\given n$. We want to give an upper bound on the probability that $\pi(\tau)\in \thickvar\LL{\kappa\rho}$, where $\LL\in\Hyp$ and $\rho > 0$. Now reveal the letters of $\tau$ in order. Each time a letter is revealed, there is a chance that the letter proves that $\tau\notin \thickvar\LL{\kappa\rho}$. More precisely, if we are revealing the $(i + 1)$st letter, and if $\tau\given i \in G$ and $D_{\tau\given i} \geq \rho$, then the probability that $\Supp(\mu_{\tau\given i + 1})$ (which contains $\tau$) is disjoint from $\thickvar\LL{\kappa\rho}$ is at least $\kappa$. So the probability that $\tau\in \thickvar\LL{\kappa\rho}$ is bounded above by an expression like $(1 - \kappa)^k$, where $k$ is the number of $i\geq n$ such that $\tau\given i \in G$ and $D_{\tau\given i} \geq \rho$. The hypotheses \eqref{klimit}-\eqref{dlimit} can be used to give a lower bound on $k$, which in turn gives an upper bound on probability which depends only on $\omega$ and $\rho$.

We now proceed to make this idea rigorous:
\begin{proof}[Proof of Theorem \ref{theoremIFSgeneral}]
For each $n\in\N$, $\rho > 0$, and $k\in\N$ let
\[
E(n,\rho,k) = \{\tau \in \alphabet^\N : \#\{i\geq n : D_{\tau\given i} \geq \rho, \; \tau\given i \in G\} \geq k\}.
\]

\begin{claim}
\label{claimmuomegainduction}
For all $\omega\in \alphabet^*$, $\LL\in\Hyp$, $\rho > 0$, and $k\in\N$,
\begin{equation}
\label{muomegainduction}
\mu_\omega\big(\NN(\LL,\kappa\rho)\cap E(|\omega|,\rho,k)\big) \leq (1 - \kappa)^k.
\end{equation}
\end{claim}
\begin{subproof}
For ease of exposition we assume that $c = \inf_{\omega} \log(D_{\omega\given |\omega| - 1}/D_\omega) > 0$, and we proceed by induction on $\lfloor (1/c)\log(D_\omega/\rho)\rfloor$. If $\rho > D_\omega$, then either $E(|\omega|,\rho,k) = \emptyset$ or $k = 0$, and in either case \eqref{muomegainduction} holds trivially. So assume that $\rho \leq D_\omega$. Write $\ell = [\omega\in G]$ (recall Convention \ref{conventioniverson}). We have
\begin{align*}
&\mu_\omega\big(\NN(\LL,\kappa\rho)\cap E(|\omega|,\rho,k)\big)\\
&= \mu_\omega\big(\NN(\LL,\kappa\rho)\cap E(|\omega| + 1,\rho,k - \ell)\big) \since{$\rho \leq D_\omega$}\\
&= \int \mu_{\tau\given |\omega| + 1}\big(\NN(\LL,\kappa\rho)\cap E(|\omega| + 1,\rho,k - \ell)\big)\;\dee \mu_\omega(\tau) \note{conditional measures}\\
&\leq \int (1 - \kappa)^{k - \ell} \big[\Supp(\mu_{\tau\given |\omega| + 1})\cap \NN(\LL,\kappa\rho)\neq\emptyset \big]\;\dee \mu_\omega(\tau) \note{induction hypothesis}\\
&\leq (1 - \kappa)^{k - \ell} \mu_\omega\big(\big\{\tau \in \alphabet^\N : \Supp(\mu_{\tau\given |\omega| + 1})\cap \NN(\LL,\kappa D_\omega)\neq\emptyset\big\}\big) \since{$\rho \leq D_\omega$}\\
&\leq (1 - \kappa)^{k - \ell} (1 - \kappa)^\ell \note{\eqref{Gkrdef} and definition of $\ell$}\\
&= (1 - \kappa)^k
\end{align*}
which completes the induction step.

The lemma can be proven without using the assumption $\inf_{\omega} \log(D_{\omega\given |\omega| - 1}/D_\omega) > 0$ by using the martingale theorem instead of induction; the essential calculations are the same and we omit the details.
\end{subproof}

Now let $\borel\subset \alphabet^\N$ be a set on which the limits \eqref{klimit}-\eqref{dlimit} converge uniformly. Fix $\gamma > 0$, $\xx\in\R^d$, $0 < \rho \leq 1$, $0 < \beta \leq \rho^\gamma$, and $\LL\in\Hyp$, and we will show that
\begin{equation}
\label{QDwithE2}
\mu\big(\thickvar\LL{\beta\rho}\cap B(\xx,\rho)\cap E\big) \lesssim_\times \beta^\alpha \mu\big(B(\xx,2\rho)\big)
\end{equation}
for some $\alpha > 0$ depending only on $\gamma$.

Consider the partition $\AA$ of $\alphabet^\N$ consisting of all cylinders $[\omega]$ ($\omega\in \alphabet^*$) satisfying
\begin{equation}
\label{partitionelement}
D_\omega < \rho \leq D_{\omega\given |\omega| - 1}.
\end{equation}
Fix such an $\omega$, let $n = |\omega|$, and let $k = \lceil \log_\lambda(\beta)\rceil$, where $\lambda\in (0,1)$ is small to be determined.
\begin{claim}
\label{claimomegaE}
If $\beta$ is sufficiently small then
\[
[\omega] \cap E \subset E(|\omega|,\kappa^{-1}\beta\rho,k).
\]
\end{claim}
\begin{subproof}
Fix $\tau\in [\omega] \cap \borel$. Let $\delta = (1/2)\log_\rho(\beta) \geq \gamma/2$, let $\epsilon = (\gamma/12)\wedge(1/5)$, and let $\ell = \lfloor \delta n\rfloor$. By \eqref{partitionelement} and the definition of $\borel$, we have\Footnote{In these asymptotics the implied constant may depend on $\kappa$ and $\gamma$ (and thus on $\epsilon$), but not on other variables such as $\delta$.}
\begin{align} \label{kbounds}
\#\{i = n + 1,\ldots,&n + \ell : \tau\given i \in G\} \gtrsim_\plus g(\delta - \epsilon)n\\ \label{sbound1}
\rho < D_{\tau\given n - 1} &\lesssim_\times \exp\big(-\chi(1 - \epsilon) n\big),\\ \label{sbound2}
D_{\tau\given n + \ell} &\gtrsim_\times \exp\big(-\chi(1 + \delta + \epsilon) n\big) \\ \label{sbound3}
\rho \geq D_{\tau\given n} &\gtrsim_\times \exp\big(-\chi (1 + \epsilon) n\big),
\end{align}
where $g,\chi > 0$ are the limits of \eqref{klimit} and \eqref{dlimit}, respectively. To show that $\tau\in E(|\omega|,\kappa^{-1}\beta\rho,k)$, it suffices to demonstrate the separate claims
\begin{align} \label{ETSrandom1}
D_{\tau\given n + \ell} &\geq \kappa^{-1} \beta\rho \\ \label{ETSrandom2}
\#\{i = n + 1,\ldots,n + \ell : \tau\given i \in G\} &\geq \lceil \log_\lambda(\beta) \rceil
\end{align}
under the assumption that $\beta$ is sufficiently small. To prove \eqref{ETSrandom1}, we observe that
\begin{align*}
\kappa\rho^{-1}D_{\tau\given n + \ell}
&\gtrsim_\times \exp\big(-\chi (\delta + 2\epsilon) n\big) \by{\eqref{sbound1} and \eqref{sbound2}}\\
&\gtrsim_\times \rho^{(\delta + 2\epsilon)/(1 - \epsilon)} \by{\eqref{sbound1}}\\
&\geq_\pt \rho^{(\delta + \gamma/6)/(1 - 1/5)} \geq \rho^{(4/3)\delta/(4/5)} = \beta^{5/6}.\noreason
\end{align*}
This implies that \eqref{ETSrandom1} holds for all sufficiently small $\beta$. To prove \eqref{ETSrandom2}, we observe that
\begin{align*}
&\#\{i = n + 1,\ldots,n + \ell : \tau\given i \in G\}\\
&\gtrsim_\plus g(\delta - \epsilon) n \by{\eqref{kbounds}}\\
&\gtrsim_\plus \frac{g(\delta - \epsilon)}{\chi (1 + \epsilon)} \log(1/\rho) \by{\eqref{sbound3}}\\
&\geq_\pt \frac{g(\delta - \gamma/12)}{\chi(1 + 1/5)} \log(1/\rho)
\geq \frac{(5/6)g}{(6/5)\chi} \delta\log(1/\rho)
> \frac{g}{3\chi}\log(1/\beta). \noreason
\end{align*}
Choosing $\lambda < \exp(-g/3\chi)$ gives \eqref{ETSrandom2} for all sufficiently small $\beta$.
\end{subproof}

Combining Claims \ref{claimmuomegainduction} and \ref{claimomegaE} yields (for $\beta$ sufficiently small)
\[
\mu_\omega\big(\NN(\LL,\beta\rho)\cap \borel\big) \leq (1 - \kappa)^k \leq \beta^\alpha,
\]
where $\alpha = \log_\lambda(1 - \kappa) > 0$. On the other hand, by \eqref{partitionelement} we have $\Diam(\Supp(\mu_\omega)) \leq D_\omega < \rho$, so either $\Supp(\mu_\omega)\cap B(\xx,\rho) = \emptyset$ or $\Supp(\mu_\omega) \subset B(\xx,2\rho)$. Either way we have
\[
\mu_\omega\big(\NN(\LL,\beta\rho)\cap B(\xx,\rho) \cap \borel\big) \leq \beta^\alpha \mu_\omega\big(B(\xx,2\rho)\big)
\]
and integrating over all $\omega$ yields \eqref{QDwithE2}. Since \eqref{QDwithE2} is an asymptotic, it holds for all $\beta$ rather than just for all sufficiently small $\beta$. Now \cite[Lemma 3.2]{DFSU_GE1} shows that for $\pi_*[\mu]$-a.e. $\xx\in\R^d$, we have
\[
\mu\big(\NN(\LL,\beta\rho)\cap B(\xx,\rho) \cap \borel\big) \leq \beta^{\alpha/2} \mu\big(B(\xx,\rho)\big).
\]
Since $\borel$ was arbitrary subject to the condition that \eqref{klimit}-\eqref{dlimit} converge uniformly, Egoroff's theorem and Lemma \ref{lemmaQDETS} show that $\pi_*[\mu]$ is quasi-decaying.
\end{proof}

\subsection{Definitions}
\label{subsectiongibbsdef}
Before proving Theorem \ref{theoremgibbs}, we recall the definition of a conformal iterated function system (CIFS), its coding map, and the Gibbs measures of summable locally H\"older continuous functions.

\begin{definition}[Cf. {\cite[p.6-7]{MauldinUrbanski1}}]
\label{definitionCIFS}
Fix $d\in\N$. A collection of maps $(u_a)_{a\in \alphabet}$ is called a \emph{conformal iterated function system} on $\R^d$ if:
\begin{enumerate}[1.]
\item $\alphabet$ is a countable (finite or infinite) index set;
\item $X\subset\R^d$ is a nonempty compact set which is equal to the closure of its interior;
\item For all $a\in \alphabet$, $u_a(X) \subset X$;
\item (Cone condition)
\[
\inf_{\xx\in X, r\in (0,1)} \frac{\lambda(X\cap B(\xx,r))}{r^d} > 0,
\]
where $\lambda$ denotes Lebesgue measure on $\R^d$;
\item $V\subset\R^d$ is an open connected bounded set such that $\dist(X,\R^d\butnot V) > 0$;
\item For each $a\in \alphabet$, $u_a$ is a conformal homeomorphism from $V$ to an open subset of $V$;
\item (Uniform contraction) $\sup_{a\in \alphabet} \sup |u_a'| < 1$, and if $\alphabet$ is infinite, $\lim_{a\in \alphabet} \sup |u_a'| = 0$;
\item (Bounded distortion property) For all $n\in\N$, $\omega\in \alphabet^n$, and $\xx,\yy\in V$,
\begin{equation}
\label{BD2}
|u_\omega'(\xx)| \asymp_\times |u_\omega'(\yy)|,
\end{equation}
where
\[
u_\omega = u_{\omega_1}\circ\cdots\circ u_{\omega_n}.
\]
\end{enumerate}
The CIFS is said to satisfy the \emph{open set condition} if the collection $(u_a(\Int(X)))_{a\in \alphabet}$ is disjoint. It is said to satisfy the \emph{strong separation condition} if the collection $(u_a(X))_{a\in \alphabet}$ is disjoint.

Finally, the CIFS is said to be \emph{irreducible} if there is no proper real-analytic submanifold $M\subset X$ such that $u_a(M)\subset M$ for all $a\in \alphabet$. (Equivalently, the limit set of the CIFS (defined below) is not contained in any proper real-analytic submanifold of $\R^d$.)
\end{definition}

In the remainder of this section, we fix a CIFS $(u_a)_{a\in \alphabet}$ and corresponding sets $X,V\subset\R^d$.

\begin{definition}
\label{definitionlimitset}
The \emph{coding map} of the CIFS $(u_a)_{a\in \alphabet}$ is the map $\pi:\alphabet^\N\to\R^d$ defined by the formula
\[
\pi(\omega) = \lim_{n\to\infty} u_{\omega\given n}(x_0),
\]
where $x_0\in X$ is an arbitrary point. By the Uniform Contraction hypothesis, $\pi(\omega)$ exists and is independent of the choice of $x_0$. The \emph{limit set} of the CIFS is the image of $\alphabet^\N$ under the coding map, i.e. $K = \pi(\alphabet^\N)$.
\end{definition}

Note that by the Uniform Contraction hypothesis, the coding map is always H\"older continuous, assuming that the metric on $\alphabet^\N$ is given by the formula
\[
\dist(\omega,\tau) = \lambda^{|\omega\wedge\tau|},
\]
where $\lambda\in (0,1)$ and $\omega\wedge\tau$ is the longest word which is an initial segment of both $\omega$ and $\tau$.

\begin{definition}[{\cite[\62.2 and \62.3]{MauldinUrbanski2}}]
\label{definitiongibbs}
A function $\phi:\alphabet^\N\to\R$ is called \emph{locally H\"older continuous} if there exist $C,\alpha > 0$ such that for all $\omega,\tau\in \alphabet^\N$ such that $\omega_1 = \tau_1$,
\[
|\phi(\omega) - \phi(\tau)| \leq C \dist^\alpha(\omega,\tau).
\]
(We note that in \cite{MauldinUrbanski2}, this was just called ``H\"older continuous'', but the terminology ``locally H\"older continuous'' is now standard, to distinguish it from the same definition with the requirement that $\omega_1 = \tau_1$ is removed.)
A locally H\"older continuous function $\phi:\alphabet^\N\to\R$ is called \emph{summable} if
\[
\sum_{a\in \alphabet} \sup_{[a]} e^\phi < \infty.
\]
A measure $\mu$ on $\alphabet^\N$ is said to be a \emph{Gibbs state} for $\phi$ if there exists $P\in\R$ such that for all $\omega\in \alphabet^*$ and $\tau\in [\omega]$,
\begin{equation}
\label{gibbs}
\mu([\omega]) \asymp_\times \exp\left(\sum_{j = 0}^{|\omega - 1|} \phi(\sigma^j \tau) - P|\omega|\right).
\end{equation}
The number $P$ is called the \emph{pressure} of $\phi$.
\end{definition}

\begin{theorem}[Special case of {\cite[Corollary 2.7.5(c) and Theorem 2.2.9]{MauldinUrbanski2}}]
\label{theoremgibbsexist}
If $\phi:\alphabet^\N\to\R$ is a summable locally H\"older continuous function, then there exists a unique Gibbs measure $\mu_\phi$ for $\phi$ which is invariant and ergodic under the shift map. Any other Gibbs measure is coarsely asymptotic to $\mu_\phi$. Moreover, $\mu_\phi$ is the unique equilibrium state for the inverse dynamical system.
\end{theorem}

\subsection{Proof of Theorem \ref{theoremgibbs}}
In what follows we will use the shorthand $\mu$ for $\mu_\phi$.
For each $\omega\in \alphabet^*$ let $D_\omega = \Diam(\pi([\omega]))$. By the Bounded Distortion Property, for all $\omega\in \alphabet^\N$ and $n\in\N$ we have
\[
D_{\omega\given n} \asymp_\times \prod_{j = 1}^n |u_{\omega_j}'(\pi\circ\sigma^j(\omega))|
\]
and so combining \eqref{lyapunovmu} with the ergodic theorem gives \eqref{dlimit} for $\mu$-a.e. $\omega\in \alphabet^\N$. Moreover, if $G = \alphabet^*$, then clearly \eqref{klimit} holds. So to complete the proof, we need to show that there exist $\kappa > 0$ and $r\in\N$ such that \eqref{Gkrdef} holds with $G = \alphabet^*$.

Suppose not; then for all $r,n\in\N$ there exist $\omega^{(r,n)} \in \alphabet^*$ and $\LL_{r,n}\in\Hyp$ such that
\begin{equation}
\label{omegarnLrn}
\mu_{\omega^{(r,n)}}\big(\big\{\tau\in \alphabet^\N : \pi([\tau\given |\omega^{(r,n)}| + r]) \cap \thickvar{\LL_{r,n}}{(1/n) D_{\omega^{(r,n)}}}\big\}\big) < 1/n.
\end{equation}
Fix $\LL_0\in\Hyp$, and for each $r,n$, let $f_{r,n}:\R^d\to\R^d$ be a similarity such that $f_{r,n}(\LL_{r,n}) = \LL_0$, $|f_{r,n}'| = 1/D_{\omega^{(r,n)}}$, and $f_{r,n}\circ u_{\omega^{(r,n)}}(x_0)$ is bounded for $x_0\in X$ fixed. Then by the Bounded Distortion Property, $(f_{r,n}\circ u_{\omega^{(r,n)}})_{r,n}$ is a normal family, and so we can find convergent subsequences
\[
f_{r,n}\circ u_{\omega^{(r,n)}} \tendsto n v_r \tendsto r v.
\]
Let $K$ be the limit set of $(u_a)_{a\in E}$. By hypothesis, $K$ is not contained in any proper real-analytic submanifold of $\R^d$; in particular, $K$ is not contained in $v^{-1}(\LL_0)$. So choose $\tau\in \alphabet^\N$ such that $v\circ\pi(\tau)\notin \LL_0$. There exists $r_0$ such that $\dist(v\circ\pi([\tau\given r_0]),\LL_0) > 0$. Since
\[
\dist\big(v_r\circ\pi([\tau\given r]),\LL_0\big) \underbrace{\geq}_{r\geq r_0} \dist\big(v_r\circ\pi([\tau\given r_0]),\LL_0\big) \tendsto r \dist\big(v\circ\pi([\tau\given r_0]),\LL_0\big) > 0,
\]
there exists $r\geq r_0$ such that $\dist(v_r\circ\pi([\tau\given r]),\LL_0) > 0$. Since
\begin{align*}
\frac{1}{D_{\omega^{(r,n)}}} \dist\big(\pi([\omega^{(r,n)}\ast \tau\given n]),\LL_{r,n}\big)
&= \dist\big(f_{r,n}\circ u_{\omega^{(r,n)}} \circ \pi([\tau\given n]),\LL_0\big)\\
&\tendsto n \dist\big(v_r\circ\pi([\tau\given r]),\LL_0\big) > 0,
\end{align*}
for all sufficiently large $n$ we have
\[
\dist\big(\pi([\omega^{(r,n)}\ast \tau\given r]),\LL_{r,n}\big) \geq (1/n) D_{\omega^{(r,n)}}.
\]
Combining with \eqref{omegarnLrn} shows that
\[
\mu_{\omega^{(r,n)}}([\omega^{(r,n)}\ast \tau\given r]) < 1/n.
\]
But by \eqref{gibbs},
\[
\mu_{\omega^{(r,n)}}([\omega^{(r,n)}\ast \tau\given r])
\asymp_\tau \exp\left(\sum_{j = 0}^{r - 1} \phi\circ\sigma^j(\tau) - Pr\right) \asymp_{\times,\tau,r} 1.
\]
Since $n$ can depend on $\tau$ and $r$, this is a contradiction.

\appendix
\section{Non-friendly quasi-decaying measures}
\label{appendix}

\subsection{An example using continued fractions}

Consider the CIFS  $(u_a)_{a\in \N}$, where the maps $u_a : [0,1] \to [0,1]$ are defined as 
\[
u_a(x) := \frac{1}{a+x}\cdot
\]
Let $\phi:\N^\N\to\R$ be the summable locally H\"older continuous function which is defined on cylinders as 
\[
\phi([n]) := -n\log(2).
\]
Then Theorem \ref{theoremgibbsexist} guarantees the existence of a unique Gibbs measure $\mu$ for $\phi$ which is invariant and ergodic under the shift map.

\begin{theorem}
\label{theoremQDnotfriendly}
The Gibbs measure $\mu$ on $\N^\N$ defined above is quasi-decaying and not friendly.
\end{theorem}
\begin{proof}
We first observe that since $\phi$ is constant on cylinders and $\sum_n e^{\phi(n)} = 1$, the invariant Gibbs measure $\mu$ is a Bernoulli measure with coefficients $(e^{\phi(n)})_{n\in\N}$. In particular, for each $n$ we have
\begin{equation}
\label{mufirstlevel}
\mu\left(\left[\tfrac1{n+1},\tfrac1n\right]\right) = \tfrac1{2^n}\cdot
\end{equation}
Summing over $n\geq N$ yields
\[
\mu\left(\left[0,\tfrac1N\right]\right) = \tfrac2{2^N}\cdot
\]
and in particular,
\[
\frac{\mu\left(\left[0,\tfrac1N\right]\right)}{\mu\left(\left[0,\tfrac1{2N}\right]\right)} = \frac{2^{2N}}{2^N} = 2^N \to \infty \text{ as } N \to \infty.
\]
This proves that $\mu$ is not Federer and thus not friendly. On the other hand, it follows from \eqref{mufirstlevel} that
\[
\chi_\mu \leq \sum_{n=1}^\infty \frac{\log(n+1)}{2^n} < \infty
\]
and thus by Theorem \ref{theoremgibbs}, $\mu$ is quasi-decaying.
\end{proof}

\subsection{An example using rational maps}
\begin{theorem}
\label{mariusz}
If $T: \hat \C \to \hat \C$ is a parabolic rational function and $\phi: J(T) \to \R$ is a H\"older continuous potential such that $P(\phi) > \sup(\phi)$, then $\mu_\phi$, the unique equilibrium state of $\phi$, fails to be Federer. Note that if in addition $J(T)$ is is not contained in a generalized sphere, then by combining with Theorem \ref{theoremrational}, we see that $\mu_{\phi}$ is quasi-decaying but not friendly.
\end{theorem}

\begin{proof}
Let $\xi \in J(T)$ be a repelling periodic point of $T$. Let $p\geq1$ be its prime period. Then 
\[
\lambda_\xi \df |(T^p)'(\xi)| >1.
\]
Then there exists $\delta>0$ so small that there exists unique holomorphic map $T_\xi^{-p}: B(\xi,\delta) \to B(\xi,\delta)$ sending $\xi$ to $\xi$ and such that $T^p \circ T_\xi^{-p}$ is the identity map restricted to $B(\xi,\delta)$. Furthermore, decreasing $\delta$ if necessary, there exists a constant $C\geq 1$ such that
\[
B(\xi, C^{-1} \lambda_\xi^{-n}\delta) \subseteq T_\xi^{-pn}(B(\xi,\delta)) \subseteq B(\xi, C \lambda_\xi^{-n}\delta)
\] 
for every integer $n \geq 0$. We note that the existence and uniqueness of $\mu_{\phi}$ is demonstrated in \cite{DenkerUrbanski}. Now let us assume, by way of contradiction, that  $\mu_{\phi}$ is Federer. It also follows from \cite{DenkerUrbanski} that
\begin{equation}
\label{1m7}
\mu_\phi \left( B(\xi, C^{-1} \lambda_\xi^{-n}\delta)  \right) 
\asymp \mu_\phi \left(  B(\xi, C \lambda_\xi^{-n}\delta) \right) 
\asymp \exp \bigg(  n \Big(S_p \phi(\xi)  - P(\phi)p \Big) \bigg) 
\end{equation}
for every integer $n \geq 0$, with all comparability constants independent of $n$. 

Since $\mu_{\phi}$ was assumed to be Federer, there exists a constant $D\geq 1$ such that 
\[
\mu_\phi(B(x,2r)) \leq D \mu_\phi(B(x,r))
\]
for all $x \in J(T)$ and all $r>0$. Denote
\begin{equation}
\label{4m7}
\kappa_\xi \df \exp\big(P(\phi)p - S_p \phi(\xi)\big).
\end{equation}
Given an integer $n \geq 1$ let $l_n \geq 1$ be the unique integer such that 
\begin{equation}
\label{2m7}
2^{l_n -1} \leq C^{-1} \lambda_\xi^n < 2^{l_n}.
\end{equation}

Then we have
\[
\mu_\phi (B(\xi,\delta)) \leq D^{l_n} \mu_\phi (B(\xi,C \lambda_\xi^{-n}\delta)). 
\]
Invoking \eqref{1m7} and \eqref{4m7}, this gives that 
\[
\mu_\phi (B(\xi,\delta)) \lesssim D^{l_n} \kappa_\xi^n \mu_\phi (B(\xi,\delta)). 
\] 
Equivalently,
\[
D^{l_n} \kappa_\xi^n \gtrsim 1.
\] 
Applying logarithms, we get 
\[
l_n \log(D) + n\log(\kappa_\xi) \gtrsim_+ 0
\]
for all integers $n \geq 1$. Dividing by $n$, and using \eqref{2m7}, and letting $n \to \infty$, we get 
\[
\log(\lambda_\xi)\frac{\log(D)}{\log(2)} + \log(\kappa_\xi) \geq 0.
\] 
Hence, 
\[
\log(\lambda_\xi)\frac{\log(D)}{\log(2)} + p \big(\sup(\phi) - P(\phi)\big) \geq 0,
\]
equivalently
\[
\left( \frac1p \log(\lambda_\xi) \right)\frac{\log(D)}{\log(2)} \geq P(\phi) - \sup(\phi) > 0.
\]
To end, note that the following claim will produce a contradiction.
\begin{claim}
We have that
\[
\inf \left\{ \frac1p \log(\lambda_\xi) \right\} = 0,
\]
where the infimum is taken over all repelling periodic points of $T$. 
\end{claim}
\begin{subproof}
Let $\eta$ be a parabolic fixed point of $T$, and let $\alpha \neq \eta$ be a preimage of $\eta$, i.e. $T(\alpha) = \eta$. Let $U$ be a neighborhood of $\eta$ small enough so that there exists a biholomorphic conjugacy $\Phi:U\to V$ between $T$ and the map $z\mapsto z+1$, where $V$ is a neighborhood of $\infty$ in the Riemann sphere (cf. \cite[Theorem 10.9]{Milnor2}).
Then there exist $n\in\N$ and $\alpha_n \in U$ such that $T^n(\alpha_n) = \alpha$. Furthermore, letting $\alpha_m = \Phi^{-1}(\Phi(\alpha_n) - (m-n))$ for $m > n$ and $\alpha_m = T^{n - m}(\alpha_n)$ for $m < n$, we see that there exists a sequence $(\alpha_m)$ such that $\alpha_0 = \alpha$, $T(\alpha_{m+1}) = \alpha_m$, and $\dist(\eta,\alpha_m) \asymp 1/m$.
Furthermore, there exists a family of neighborhoods $(B_m)_{m\geq n}$ of $(\alpha_m)_{m\geq n}$, such that $T(B_{m+1}) = B_m$.
Let $B := B_n$. We then have that $T^{n+1}(B) \ni \eta$. Therefore, by shrinking $B$ if necessary, without loss of generality we may assume that 
\[
\eta \notin \del T^{n+1}(B).
\]
Therefore, it follows that the winding number $\omega(\del T^{n+1}(B),\eta) > 0$.
Let $W$ be the connected component of the complement of the boundary of $ \C \butnot \del T^{n+1}(B)$ containing $\eta$. We know that $\omega(\del T^{n+1}(B),\cdot) >0$ on $W$. Thus, for $m$ large enough, we have that $B_m \subset W$ and thus $B_m$ contains a periodic point of order $m+1$, which we call $\xi$. It follows from a straightforward computation that $|(T^{m+1})'(\xi)| \asymp m^2$, and in particular $\xi$ is repelling for $m$ sufficiently large.
Since
\[
\lim_{m \to \infty} \frac{1}{m+1} \log(m^2) = 0
\] 
this ends the proof of the claim.
\end{subproof}
Thus concludes the proof of Theorem \ref{mariusz}. 
\end{proof}

\bibliographystyle{amsplain}

\bibliography{bibliography}

\providecommand{\bysame}{\leavevmode\hbox to3em{\hrulefill}\thinspace}
\providecommand{\MR}{\relax\ifhmode\unskip\space\fi MR }
\providecommand{\MRhref}[2]{%
  \href{http://www.ams.org/mathscinet-getitem?mr=#1}{#2}
}
\providecommand{\href}[2]{#2}
\begin{thebibliography}{10}

\bibitem{Barany}
Bal\'{a}zs B\'{a}r\'{a}ny, \emph{On the {H}ausdorff dimension of a family of
  self-similar sets with complicated overlaps}, Fund. Math. \textbf{206}
  (2009), 49--59. \MR{2576260}

\bibitem{BPS}
Luis~Manuel Barreira, Yakov~B. Pesin, and J\"{o}rg Schmeling, \emph{Dimension
  and product structure of hyperbolic measures}, Ann. of Math. (2) \textbf{149}
  (1999), no.~3, 755--783. \MR{1709302}

\bibitem{BarrioJimenez}
Alejo Barrio~Blaya and V\'{\i}ctor Jim\'{e}nez~L\'{o}pez, \emph{On the
  relations between positive {L}yapunov exponents, positive entropy, and
  sensitivity for interval maps}, Discrete Contin. Dyn. Syst. \textbf{32}
  (2012), no.~2, 433--466. \MR{2837068}

\bibitem{BergweilerEremenko}
Walter Bergweiler and Alexandre Eremenko, \emph{Meromorphic functions with
  linearly distributed values and {J}ulia sets of rational functions}, Proc.
  Amer. Math. Soc. \textbf{137} (2009), no.~7, 2329--2333. \MR{2495266}

\bibitem{Bowen3}
Rufus Bowen, \emph{Periodic points and measures for {A}xiom {$A$}
  diffeomorphisms}, Trans. Amer. Math. Soc. \textbf{154} (1971), 377--397.
  \MR{282372}

\bibitem{BFKRW}
Ryan Broderick, Lior Fishman, Dmitry Kleinbock, Asaf Reich, and Barak Weiss,
  \emph{The set of badly approximable vectors is strongly {$C^1$}
  incompressible}, Math. Proc. Cambridge Philos. Soc. \textbf{153} (2012),
  no.~2, 319--339. \MR{2981929}

\bibitem{BFS1}
Ryan Broderick, Lior Fishman, and David Simmons, \emph{Badly approximable
  systems of affine forms and incompressibility on fractals}, J. Number Theory
  \textbf{133} (2013), no.~7, 2186--2205. \MR{3035957}

\bibitem{DFSU_GE1}
Tushar Das, Lior Fishman, David Simmons, and Mariusz Urba\'{n}ski,
  \emph{Extremality and dynamically defined measures, part {I}: {D}iophantine
  properties of quasi-decaying measures}, Selecta Math. (N.S.) \textbf{24}
  (2018), no.~3, 2165--2206. \MR{3816502}

\bibitem{DasSimmons1}
Tushar Das and David Simmons, \emph{The {H}ausdorff and dynamical dimensions of
  self-affine sponges: a dimension gap result}, Invent. Math. \textbf{210}
  (2017), no.~1, 85--134. \MR{3698340}

\bibitem{DSU_rigidity}
Tushar Das, David Simmons, and Mariusz Urba\'{n}ski, \emph{Dimension rigidity
  in conformal structures}, Adv. Math. \textbf{308} (2017), 1127--1186.
  \MR{3600084}

\bibitem{DSU}
Tushar Das, David Simmons, and Mariusz Urba\'nski, \emph{Geometry and dynamics
  in {G}romov hyperbolic metric spaces}, Mathematical Surveys and Monographs,
  vol. 218, American Mathematical Society, Providence, RI, 2017, With an
  emphasis on non-proper settings. \MR{3558533}

\bibitem{DenkerUrbanski}
Manfred Denker and Mariusz Urba{\'n}ski, \emph{Ergodic theory of equilibrium
  states for rational maps}, Nonlinearity \textbf{4} (1991), no.~1, 103--134.
  \MR{1092887}

\bibitem{EremenkoVanstrien}
Alexandre Eremenko and Sebastian~J. van Strien, \emph{Rational maps with real
  multipliers}, Trans. Amer. Math. Soc. \textbf{363} (2011), no.~12,
  6453--6463. \MR{2833563}

\bibitem{FSU1}
Lior Fishman, David Simmons, and Mariusz Urba\'{n}ski, \emph{Diophantine
  properties of measures invariant with respect to the {G}auss map}, J. Anal.
  Math. \textbf{122} (2014), 289--315. \MR{3183529}

\bibitem{Hofbauer2}
Franz Hofbauer, \emph{Generic properties of invariant measures for continuous
  piecewise monotonic transformations}, Monatsh. Math. \textbf{106} (1988),
  no.~4, 301--312. \MR{973140}

\bibitem{Hofbauer}
\bysame, \emph{Local dimension for piecewise monotonic maps on the interval},
  Ergodic Theory Dynam. Systems \textbf{15} (1995), no.~6, 1119--1142.
  \MR{1366311}

\bibitem{InoquioRivera}
Irene Inoquio-Renteria and Juan Rivera-Letelier, \emph{A characterization of
  hyperbolic potentials of rational maps}, Bull. Braz. Math. Soc. (N.S.)
  \textbf{43} (2012), no.~1, 99--127. \MR{2909925}

\bibitem{KLW}
Dmitry~Ya. Kleinbock, Elon Lindenstrauss, and Barak Weiss, \emph{On fractal
  measures and {D}iophantine approximation}, Selecta Math. (N.S.) \textbf{10}
  (2004), no.~4, 479--523. \MR{2134453}

\bibitem{KleinbockMargulis}
Dmitry~Ya. Kleinbock and Grigori\u{\i}~Aleksandrovitch Margulis,
  \emph{Logarithm laws for flows on homogeneous spaces}, Invent. Math.
  \textbf{138} (1999), no.~3, 451--494. \MR{1719827}

\bibitem{KleinbockWeiss1}
Dmitry~Ya. Kleinbock and Barak Weiss, \emph{Badly approximable vectors on
  fractals}, Israel J. Math. \textbf{149} (2005), 137--170, Probability in
  mathematics. \MR{2191212}

\bibitem{MSU}
R.~Daniel Mauldin, Tomasz Szarek, and Mariusz Urba\'{n}ski, \emph{Graph
  directed {M}arkov systems on {H}ilbert spaces}, Math. Proc. Cambridge Philos.
  Soc. \textbf{147} (2009), no.~2, 455--488. \MR{2525938}

\bibitem{MauldinUrbanski1}
R.~Daniel Mauldin and Mariusz Urba{\'n}ski, \emph{Dimensions and measures in
  infinite iterated function systems}, Proc. London Math. Soc. (3) \textbf{73}
  (1996), no.~1, 105--154. \MR{1387085}

\bibitem{MauldinUrbanski2}
R.~Daniel Mauldin and Mariusz Urba\'{n}ski, \emph{Graph directed {M}arkov
  systems}, Cambridge Tracts in Mathematics, vol. 148, Cambridge University
  Press, Cambridge, 2003, Geometry and dynamics of limit sets. \MR{2003772}

\bibitem{Milnor2}
John Milnor, \emph{Dynamics in one complex variable}, third ed., Annals of
  Mathematics Studies, vol. 160, Princeton University Press, Princeton, NJ,
  2006. \MR{2193309}

\bibitem{Patterson2}
Samuel~James Patterson, \emph{The limit set of a {F}uchsian group}, Acta Math.
  \textbf{136} (1976), no.~3-4, 241--273. \MR{450547}

\bibitem{PPS}
Fr\'{e}d\'{e}ric Paulin, Mark Pollicott, and Barbara Schapira,
  \emph{Equilibrium states in negative curvature}, Ast\'{e}risque (2015),
  no.~373, viii+281. \MR{3444431}

\bibitem{PollicottSimon}
Mark Pollicott and K\'{a}roly Simon, \emph{The {H}ausdorff dimension of
  {$\lambda$}-expansions with deleted digits}, Trans. Amer. Math. Soc.
  \textbf{347} (1995), no.~3, 967--983. \MR{1290729}

\bibitem{PollingtonVelani}
Andrew~Douglas Pollington and Sanju~L. Velani, \emph{Metric {D}iophantine
  approximation and ``absolutely friendly'' measures}, Selecta Math. (N.S.)
  \textbf{11} (2005), no.~2, 297--307. \MR{2183850}

\bibitem{PrzytyckiUrbanski}
Feliks Przytycki and Mariusz Urba\'{n}ski, \emph{Conformal fractals: ergodic
  theory methods}, London Mathematical Society Lecture Note Series, vol. 371,
  Cambridge University Press, Cambridge, 2010. \MR{2656475}

\bibitem{Rivera}
Juan Rivera-Letelier, \emph{The maximal entropy measure detects non-uniform
  hyperbolicity}, Math. Res. Lett. \textbf{17} (2010), no.~5, 851--866.
  \MR{2727614}

\bibitem{Sigmund}
Karl Sigmund, \emph{Generic properties of invariant measures for {A}xiom {${\rm
  A}$} diffeomorphisms}, Invent. Math. \textbf{11} (1970), 99--109. \MR{286135}

\bibitem{Simmons1}
David Simmons, \emph{Conditional measures and conditional expectation;
  {R}ohlin's disintegration theorem}, Discrete Contin. Dyn. Syst. \textbf{32}
  (2012), no.~7, 2565--2582. \MR{2900561}

\bibitem{SimonSolomyak}
K\'{a}roly Simon and Boris Solomyak, \emph{On the dimension of self-similar
  sets}, Fractals \textbf{10} (2002), no.~1, 59--65. \MR{1894903}

\bibitem{StratmannUrbanski1}
Bernd~O. Stratmann and Mariusz Urba\'{n}ski, \emph{Diophantine extremality of
  the {P}atterson measure}, Math. Proc. Cambridge Philos. Soc. \textbf{140}
  (2006), no.~2, 297--304. \MR{2212281}

\bibitem{StratmannVelani}
Bernd~Otto Stratmann and Sanju~L. Velani, \emph{The {P}atterson measure for
  geometrically finite groups with parabolic elements, new and old}, Proc.
  London Math. Soc. (3) \textbf{71} (1995), no.~1, 197--220. \MR{1327939}

\bibitem{Sullivan_density_at_infinity}
Dennis~P. Sullivan, \emph{The density at infinity of a discrete group of
  hyperbolic motions}, Inst. Hautes \'{E}tudes Sci. Publ. Math. (1979), no.~50,
  171--202. \MR{556586}

\bibitem{Sullivan_entropy}
\bysame, \emph{Entropy, {H}ausdorff measures old and new, and limit sets of
  geometrically finite {K}leinian groups}, Acta Math. \textbf{153} (1984),
  no.~3-4, 259--277. \MR{766265}

\bibitem{SUZ1}
Micha\l Szostakiewicz, Mariusz Urba\'{n}ski, and Anna Zdunik, \emph{Fine
  inducing and equilibrium measures for rational functions of the {R}iemann
  sphere}, Israel J. Math. \textbf{210} (2015), no.~1, 399--465. \MR{3430280}

\bibitem{Urbanski}
Mariusz Urba\'{n}ski, \emph{Diophantine approximation and self-conformal
  measures}, J. Number Theory \textbf{110} (2005), no.~2, 219--235.
  \MR{2122607}

\bibitem{Urbanski3}
\bysame, \emph{Diophantine approximation for conformal measures of
  one-dimensional iterated function systems}, Compos. Math. \textbf{141}
  (2005), no.~4, 869--886. \MR{2148198}

\end{thebibliography}

\end{document}